\documentclass[11pt]{article}
\usepackage[margin=1in]{geometry}
\usepackage{amsmath,amssymb,amsthm}
\usepackage{graphicx}
\usepackage{hyperref}

\theoremstyle{thmstyleone}%
\newtheorem{theorem}{Theorem}
\newtheorem{proposition}{Proposition}%
\newtheorem{lemma}{Lemma}
\newtheorem{assumption}{Assumption}
    
\newcommand{\pos}[1]{\bigl[#1\bigr]_{+}}
\newcommand{\scp}[2]{\langle #1, #2 \rangle}

\theoremstyle{thmstyletwo}%

\theoremstyle{thmstylethree}%
\newtheorem{definition}{Definition}%

\usepackage{graphicx}
\usepackage{epstopdf}
\usepackage{anyfontsize}
\usepackage{enumitem}
\usepackage{booktabs}
\usepackage{appendix}
\usepackage{algorithm}
\usepackage{algpseudocode}
\usepackage{float}

\begin{document}

\title{Descent-Net: Learning Descent Directions for Constrained Optimization}








\author{Zisheng Zhou, Dengyu Zheng, Zirui Chen and Shixiang Chen
\thanks{
School of Mathematical Sciences, Key Laboratory of the Ministry of Education for Mathematical Foundations and Applications of Digital Technology, University of Science and Technology of China, Hefei, Anhui, China}}


\maketitle

\begin{abstract}
Deep learning approaches, known for their ability to model complex relationships and fast execution, are increasingly being applied to solve large optimization problems. However, existing methods often face challenges in simultaneously ensuring feasibility and achieving an optimal objective value. To address this issue, we propose Descent-Net, a neural network designed to learn an effective descent direction from a feasible solution. By updating the solution along this learned direction, Descent-Net improves the objective value while preserving feasibility. Our method demonstrates strong performance on both synthetic optimization tasks and the real-world AC optimal power flow problem, while also exhibiting effective scalability to large problems, as shown by portfolio optimization experiments with thousands of assets.
\end{abstract}


\section{Introduction}\label{sec1}

Constrained optimization problems are ubiquitous in practical applications. While traditional optimization algorithms \cite{luenberger1984linear,nocedal1999numerical} offer strong theoretical guarantees, their computational efficiency often falls short when applied to modern large-scale problems. As a result, there is increasing interest in leveraging neural network-based methods to tackle constrained optimization tasks. In recent years, many emerging works have proposed end-to-end frameworks for solving constrained optimization problems \cite{donti2017task,amos2017optnet,zhang2018ista,agrawal2019differentiable,geng2020coercing}.
 
This research direction falls under the broader framework of Learning to Optimize (L2O)\cite{bengio2021machine,chen2022learning}, which aims to leverage deep learning to improve the efficiency and scalability of optimization algorithms. Unlike traditional methods that rely on handcrafted update rules, L2O methods attempt to automatically learn optimization behaviors through data-driven approaches. However, most existing works consider unconstrained optimization problems. This motivates the development of more flexible frameworks that can incorporate feasibility into the learning dynamics while remaining scalable to large or structured problems. 

In this work, we propose Descent-Net, a neural network architecture that takes as input the gradients of both the objective and constraint functions at a given feasible point. The network is trained to predict an effective descent direction and an appropriate step size, enabling objective improvement while maintaining feasibility. Initialized from feasible solutions obtained by methods such as DC3~\cite{dontidc3}, H-proj~\cite{liang2024homeomorphic}, etc., our method typically converges to a near-optimal solution in just a few update steps.

Our main contributions are summarized as follows:
\begin{itemize}
\item We design a new exact penalty subproblem that generates feasible descent directions for linearly constrained optimization problems, forming the foundation of our approach with theoretical convergence guarantees.

\item We propose a neural network architecture, \textbf{Descent-Net}, which unrolls a projected subgradient method to solve the proposed subproblem. The network iteratively refines feasible solutions by learning effective descent directions at each step.

\item We validate the effectiveness and scalability of Descent-Net through extensive experiments on quadratic programs, a simple nonconvex variant of QP with linear constraints, and the nonlinear AC optimal power flow problem. The method scales to QP instances with up to 5000 variables and to large-scale portfolio optimization tasks with up to 4000 assets, consistently producing high-quality feasible solutions with low relative errors and strong computational efficiency.


\end{itemize}

\section{Related work}
\textbf{Classical methods for constrained optimization.}
Classical approaches to constrained optimization, including projected gradient descent, feasible direction methods \cite{zoutendijk1960methods,topkis1967convergence}, and primal-dual algorithms \cite{luenberger1984linear,nocedal1999numerical,boyd2011distributed}, have been extensively studied and widely applied. These methods typically offer convergence guarantees under suitable assumptions, but often suffer from high iteration complexity and significant computational cost. Recently, GPU-accelerated algorithms have also been developed, such as \verb|HPR-QP| \cite{chen2025hpr} and \verb|CuClarabel| \cite{CuClarabel}. \\
\textbf{Learning to optimize (L2O).}
L2O seeks to replace hand-crafted optimization routines with learnable architectures that generalize across problem instances. Broadly speaking, L2O methods can be classified into model-free and model-based approaches \cite{chen2022learning}. Model-free methods, such as those based on recurrent neural networks (e.g., LSTM) \cite{graves2014generating,andrychowicz2016learning}, aim to learn update rules directly from data. Model-based methods, on the other hand, incorporate algorithmic structure into the design of the network. Notable examples include LISTA \cite{chen2018theoretical}, unrolled manifold optimization algorithms \cite{gao2022learning}. However, most existing L2O methods focus on unconstrained or simple constrained  problems and fail to guarantee feasibility when applied to general constrained settings.

To address this, recent works incorporate constrained optimization structures into neural networks via projection layers \cite{yangprojection,liang2024homeomorphic} or differentiable optimization modules \cite{amos2017optnet,agrawal2019differentiable,bolte2021nonsmooth}. However, these methods typically suffer from scalability and the need to solve nested optimization problems during training. Some approaches target special cases, such as linear constraints \cite{wang2023linsatnet}, but their applicability to more general problems remains limited. An alternative line of work draws inspiration from primal-dual methods, leading to neural architectures based on ADMM \cite{xie2019differentiable} and PDHG \cite{li2024pdhg}. Such methods are usually evaluated by the KKT error, where feasibility and objective optimality are of the same order of magnitude, which makes them less suitable for scenarios requiring strict constraint satisfaction. Recent efforts have attempted to address this by designing networks that explicitly return feasible points \cite{dontidc3,wu2025constraint}; however, such methods still fall short of reaching near-optimal solutions in practice.\\
\textbf{Implicit layers.}
A growing body of work explores the use of implicit neural architectures, including optimization layers \cite{amos2017optnet}, neural ordinary differential equations (ODEs) \cite{chen2018neural}, and deep equilibrium models (DEQs) \cite{bai2019deep}. These models define network outputs via the solution of fixed-point or optimization problems, allowing compact yet highly expressive representations. Despite their potential, these approaches often incur high computational costs during both forward and backward passes. In the context of constrained optimization, additional challenges arise when estimating gradients of projection operators, particularly in the presence of complex or nonconvex constraints. Approximate techniques such as gradient perturbation or stochastic sampling \cite{poganvcic2019differentiation, berthet2020learning} have been proposed, but typically come at the expense of increased variance and computational overhead.

\section{Problem setup }

For any given data $x\in\mathbb{R}^d$, we solve the following constrained optimization problem
\begin{equation}\label{prob:constraint}
    \min_{y \in \mathbb{R}^n} f_x(y), \quad \text{s.t.}\quad y \in \mathcal{C}:=\{y\mid g_x(y) \leqslant 0, \: h_x(y) = 0\},
\end{equation}
where $f_x, g_x$, and $h_x$ are smooth (but not necessarily convex) functions that may depend on the input data $x$. We assume that there are $m$ equality constraints and $l$ inequality constraints:
\begin{align*}
    h_x(y) &= [h_{x,1}(y), h_{x,2}(y),\cdots, h_{x,m}(y)]^T = 0,\\
    g_x(y) &= [g_{x,1}(y), g_{x,2}(y),\cdots, g_{x,l}(y)]^T \leqslant 0,
\end{align*}
where $h_{x,i}:\mathbb{R}^n\rightarrow\mathbb{R}$ and $g_{x,j}:\mathbb{R}^n\rightarrow\mathbb{R}$ for all $i = 1,\cdots, m$ and $j = 1,\cdots, l$. We have the following common assumptions for this problem.
\begin{assumption}\label{assump:close-bounded}
     The feasible set
  $\mathcal{C}$ is non-empty and closed; the sub-level set $\{y\in\mathcal{C}\mid f_x(y)\leqslant f_x(y_0)\}$ is bounded.
\end{assumption} 
\begin{assumption}\label{assump:linear_indep_eq}
     We assume that at any feasible point $y$, the gradients of the equality constraints, $\nabla h_i(y)$, for 
 $i=1,2,\cdots, m$, are linearly independent.
\end{assumption}
We also assume that the Linear Independence Constraint Qualification (LICQ) holds, which guarantees that the Karush–Kuhn–Tucker (KKT) conditions are necessary for local optimality.

\begin{assumption}[LICQ]
\label{assump:licq}
Let \( y^* \in \mathcal{C} \) be a local optimal point  of problem \eqref{prob:constraint}.
We assume that the set of active constraint gradients at \( y^* \),
\[
\left\{ \nabla h_i(y^*) \right\}_{i=1}^m
\;\cup\;
\left\{ \nabla g_j(y^*) \right\}_{j \in \mathcal{A}(y^*)},
\quad
\text{where } \mathcal{A}(y^*):= \{ j \in \{1,\cdots,l\} \mid g_j(y^*) = 0 \},
\]
is linearly independent.
\end{assumption}
 The notation $\mathcal{A}$ denotes the \textit{active set}\footnote{This is distinct from the standard definition of the active set, which typically includes the indices corresponding to the equality constraints. }, i.e., the set of inequality constraints that are satisfied with equality.
 
\begin{assumption}\label{assump:Uniform_boundedness}
    Let $\mathcal{X} \subseteq \mathbb{R}^p$ be a compact set and assume that all training and test parameters satisfy $x \in \mathcal{X}$. For each $x \in \mathcal{X}$, consider the feasible set $\mathcal{C}$. We assume that: 
    \begin{enumerate}
        \item ({Uniform boundedness of feasible sets}) There exists a compact set $Y \subseteq \mathbb{R}^n$ such that  
       \[
       \mathcal{C} \subseteq Y \quad \text{for all } x \in \mathcal{X}.
       \]
       \item ({Smoothness and uniform gradient bound})  
       The functions $f_x, h_x, g_x$ are continuously differentiable in $y$, and the maps  
       \[
       (x, y) \mapsto \nabla_y f_x(y)\quad \text{and} \quad (x, y) \mapsto \nabla_y g_x(y)
       \]  
       are continuous on $\mathcal{X} \times Y$.  
       Then, by compactness, there exist constants \(L_f > 0\) and \(L_g > 0\) such that  
       \[
       \| \nabla_y f_x(y) \|_2 \leqslant L_f \quad \text{and} \quad \| \nabla_y g_x(y) \|_2 \leqslant L_g \quad \text{for all } x \in \mathcal{X}, \, y \in \mathcal{C}.
       \]
    \end{enumerate}
\end{assumption}
This assumption is reasonable, since practical training always involves a finite dataset, guaranteeing the existence of an appropriate upper bound.

\begin{assumption}\label{assump:margin}
   There exists a constant $\delta > 0$ such that for every $x \in \mathcal{X}$ and every feasible point $y \in \mathcal{C}$,  
   \[
   \min_{j: g_{x,j}(y) < 0} \left( -g_{x,j}(y) \right) \geqslant \delta_g,
   \]  
   with the convention that the minimum over an empty index set is $+\infty$ (i.e., when all inequality constraints are active).
\end{assumption}
In fact, this assumption is not restrictive. In practice, one may choose $\delta_g$ to be a very small constant (e.g., $10^{-5}$) and treat a constraint as active whenever $0 \leqslant -g_{x,j}(y) < \delta_g$.

\subsection{Feasible directions method}
The method of feasible directions (MFD) was originally developed by Zoutendijk in the 1960s \cite{zoutendijk1960methods}. However, a well-known drawback of MFD is that it may fail to converge due to the so-called jamming phenomenon. To address this issue, various fundamental modifications and extensions of MFD have since been proposed and studied \cite{zoutendijk1960methods,topkis1967convergence,pironneau1973rate,luenberger1984linear}. 
In this section, we briefly review the framework of MFD under the assumption that the constraints  $h_x(y)$ and $g_x(y)$ are linear.
 
Based on the first-order approximation of the constraint functions, it can be inferred that, to maintain the feasibility of the solution, a suitable descent direction $d$ at the current iterate $y$ should satisfy the following conditions:
\begin{equation}\label{6}
\begin{aligned}
    \langle d,\nabla h_{x,i}(y)\rangle&=0,\quad\text{for } i=1,\cdots,m,\\
    \langle d,\nabla g_{x,j}(y)\rangle&\leqslant0,\quad\text{for } j\in \mathcal{A}=\{1\leqslant j\leqslant l:\:g_{x,j}(y)=0\},
\end{aligned}
\end{equation}
where $\nabla h_{x,i}$ denotes the gradient of the equality constraints and $\nabla g_{x,j}$ corresponds to the inequality constraints. 

\paragraph{Zoutendijk Direction-Finding Subproblem. \cite{zoutendijk1960methods}} 
 The Zoutendijk method computes a search direction $d \in \mathbb{R}^n$ at a feasible point $y$ by solving the following linear program:
\begin{equation}\label{eq:zoutendijk}\tag{MFD}
\begin{aligned}
\min_{d \in \mathbb{R}^n} \quad & \nabla f_x(y)^\top d \\
\text{s.t.} \quad & \nabla h_x(y)^\top d = 0, \quad  \nabla g_{x,j}(y)^\top d \leqslant 0, \quad j \in \mathcal{A}, \quad \|d\|_\infty \leqslant 1.
\end{aligned}
\end{equation}
Here, $\nabla h_x(y)^\top \in \mathbb{R}^{m \times n}$ denotes the Jacobian matrix of equality constraints.

The first constraint ensures that the direction is tangent to the equality constraint, while the second maintains feasibility with respect to the active inequalities. The infinity norm constraint serves to normalize the direction and keep the subproblem bounded.

The step size is then chosen as the largest feasible value such that $y + \alpha d$ remains in the feasible set:
\[
\bar{\alpha} = \max \{ \alpha \in (0,1] \mid y + \alpha d \in \mathcal{C} \}.
\]
However, when the iterate approaches the boundary of the feasible region, the step size $\bar{\alpha}$ may become arbitrarily small, potentially causing convergence issues \cite{topkis1967convergence}.

\paragraph{Topkis--Veinott Uniformly Feasible Direction Subproblem \cite{topkis1967convergence}} 
To resolve this issue, Topkis and Veinott proposed a uniformly feasible direction (UFD) formulation:
\begin{equation}\label{eq:Topkis}\tag{UFD}
\begin{aligned}
\min_{d \in \mathbb{R}^n} \quad & \nabla f_x(y)^\top d \\
\text{s.t.} \quad & \nabla h_x(y)^\top d = 0, \quad  \nabla g_{x,j}(y)^\top d \leqslant - M \cdot g_{x,j}(y), \quad j = 1,\cdots,l, \\
& \sum_{i=1}^n \lvert d_i\rvert = 1.
\end{aligned}
\end{equation}
The main difference lies in the fact that all inequality constraints are considered, and a constant $M > 0$ is introduced. Notably, setting $M=\infty$ recovers the original formulation in \eqref{eq:zoutendijk}. This modification ensures that the computed direction  $d$ satisfies
\[
g_{x,j}(y + \alpha d)\leqslant 0, \quad \text{for all } j,
\quad \text{as long as } \alpha \leqslant \frac{1}{M},
\]
thus providing a uniform lower bound on feasible step sizes and overcoming the stalling issues of the original method. It can be shown that the direction obtained from \eqref{eq:Topkis} is a feasible descent direction. Moreover, under the Assumption \ref{assump:licq}, any accumulation point of the iterates generated by this method \cite{zoutendijk1960methods,faigle2013algorithmic} satisfies the KKT conditions.

\section{Algorithm}

\subsection{Reformulation of UFD subproblem} 

Our goal is to design a learning-to-optimize (L2O) algorithm for solving the structured problem described above. However, both the Zoutendijk and Topkis--Veinott methods require solving  constrained subproblems at each iteration, which are not suitable for direct embedding into neural networks. To address this, we reformulate the subproblem by exact penalty method. This enables us to implement the solver as an unrolled optimization process of projected subgradient method, forming the basis of our L2O algorithm. In the following, we describe the reformulated subproblem and the corresponding L2O architecture.

Motivated by \eqref{eq:zoutendijk}, we first formulate the following penalized subproblem:
\begin{equation}\label{eq:subproblem-1}
\begin{aligned}
    \min_{d}  \nabla f_x(y)^\top d+ \sum_{j=1}^l \ c_j \max\left(\langle d,\nabla g_{x,j}(y)\rangle,-M g_{x,j}(y)\right), \quad \mathrm{s.t.} \quad d\in\mathcal{D},
\end{aligned}
\end{equation}
where $\mathcal{D}=\{d:\|d\|_2\leqslant 1\text{ and } \langle d,\nabla h_{x,i}(y)\rangle=0,\:\forall i=1,\cdots, m\}$, $c_j>0$ and $M>0$ are the regularization parameters. The hinge penalty is exact if the parameter $c_j$ is large enough.

The hinge penalty is exact if the parameter $c_j$ is large enough.

\begin{lemma}[Exact hinge penalty]
\label{lem:exact_penalty_l2}
Given any feasible point $y\in \mathcal{C}$, denote $c_{\min}:=\min_j c_j$.  
If we have
\begin{equation}\label{eq:lambda-cond-main}
   c_{\min} >\frac{L_f}{M\delta_g},
\end{equation}
where $c_{\min}$ is selected independently of $x$, then every global minimizer of \eqref{eq:subproblem-1} is optimal for the following $L_2$-norm UFD subproblem 
\begin{equation}\label{eq:Topkis-2}\tag{UFD-L2}
\begin{aligned}
\min_{\|d\|_2\leqslant 1} \quad & \nabla f_x(y)^\top d \\
\text{s.t.} \quad & \nabla h_x(y)^\top d = 0, \quad  \nabla g_{x,j}(y)^\top d \leqslant - M \cdot g_{x,j}(y), \quad j = 1,\cdots,l. 
\end{aligned}
\end{equation}
\end{lemma}

To prove Lemma~\ref{lem:exact_penalty_l2}, we introduce the following notations:
\begin{itemize}
  \item Gradient of the linearized objective: $p:=\nabla f_x(y)\in\mathbb{R}^{n}$.
  \item For each constraint $j$: 
        $a_j:=\nabla g_{x,j}(y)\in\mathbb{R}^{n}$ and
        $b_j:=-M\,g_{x,j}(y)$.
  \item Linear-equality matrix: $E:=\nabla h_x(y)^{\top}\in\mathbb{R}^{m\times n}$.         Let $P:=I-E^{\top}(EE^{\top})^{-1}E$ denote the orthogonal projector onto $\ker(E)$.
  \item Feasible search set:
        \[
          \mathcal D:=\{d\in\mathbb{R}^{n}\mid Ed=0,\ \|d\|_2\leqslant 1\}.
        \]
  \item UFD feasible region:
        \[
          F:=\{d\in\mathcal D:\ \langle a_j,d\rangle\leqslant b_j,\ j=1,\cdots,l\}.
        \]
\end{itemize}

Under this notation, the penalized subproblem \eqref{eq:subproblem-1} becomes  
\begin{equation}\label{eq:pen}
   \min_{d\in\mathcal D}\;
   \Phi(d)
   :=\langle p,d\rangle
     +\sum_{j=1}^l c_j\max\{\langle a_j,d\rangle,b_j\},
   \qquad c_j>0,
\tag{Pen}
\end{equation}
while the UFD subproblem \eqref{eq:Topkis-2} can be written as
\begin{equation}\label{eq:UFD}
   \min_{d\in F}\langle p,d\rangle.
\tag{UFD-L2}
\end{equation}
Thus, in the new notation, Lemma~\ref{lem:exact_penalty_l2} asserts that if
$c_{\min}:=\min_j c_j$ satisfies
\[
     c_{\min} > \frac{L_f}{M\delta_g},
\]
then every global minimizer of \eqref{eq:pen} must lie in $F$, and therefore
the two problems share the same minimizers.

\begin{proof}
Let 
\[
   \tilde{L}:= \|\nabla f_x(y)\|_2,\qquad b_{\min}:=\min_{j:b_j>0}b_j.
\]
By assumptions (\ref{assump:Uniform_boundedness}) and (\ref{assump:margin}), we have $\tilde{L}\leqslant L_f$ and $b_{\min}\geqslant M\delta_g$. Hence 
\[
c_{\min} >\frac{L_f}{M\delta_g}\geqslant \frac{\tilde{L}}{b_{\min}}.
\]

Suppose $d\in\mathcal D$ is the optimal point of problem \eqref{eq:pen}, but $d\notin F.$ Define the violation vector $r(d):=(\pos{\scp{a_1}{d}-b_1},\dots,\pos{\scp{a_l}{d}-b_l})\in\mathbb{R}_{\geqslant0}^l$ and let $V(d):=\{j\mid r_j(d)>0\}$ be the index set of violated constraints.

If $r(d)=0$, then $d\in F$. Otherwise, define
\[
\alpha(d):=\min_{j\in V(d)}\frac{b_j}{\langle a_j, d\rangle}\in(0,1),
\qquad
\hat d:=\alpha(d)\,d .
\]
Since $Ed=0$ and $\alpha(d)\leqslant 1$, we have $\hat d\in\mathcal D$.
Moreover, for every $j$, $\langle a_j, \hat d\rangle=\alpha(d)\langle a_j, d\rangle\leqslant b_j$, so $\hat d\in F$.

Select $\bar\jmath\in V(d)$ that attains the minimum in $\alpha(d)$ and
set $\delta:=r_{\bar\jmath}(d)=\langle a_{\bar\jmath}, d\rangle-b_{\bar\jmath}>0$.
Then
\[
     1-\alpha(d)
     = 1-\frac{b_{\bar\jmath}}{\langle a_{\bar\jmath}, d\rangle}
     = \frac{\delta}{\langle a_{\bar\jmath}, d\rangle}
     \leqslant\frac{\delta}{b_{\bar\jmath}}
     \leqslant\frac{\delta}{b_{\min}},
\]
where we used $\langle a_{\bar\jmath}, d\rangle>b_{\bar\jmath}$.
Since $\|d\|_2\leqslant1$, we obtain
\[
     \Vert d-\hat d\Vert_2 = (1-\alpha(d))\Vert d\Vert_2
     \leqslant\frac{\delta}{b_{\min}}.
\]

The linear part is $\tilde{L}$-Lipschitz on $\mathcal D$:
\[
   \bigl\lvert \langle p, d\rangle-\langle p, \hat d\rangle\bigr\rvert
    \leqslant \tilde{L}\Vert d-\hat d\Vert_2
    \leqslant \frac{\tilde{L}}{b_{\min}}\;\delta.
\]
Since $\hat d\in F$, we have $\max\{\langle a_j, \hat d\rangle,b_j\}=b_j$ for all $j$, while for $d$
\[
     \max\{\langle a_j, d\rangle,b_j\}-b_j=\pos{\langle a_j, d\rangle-b_j}=r_j(d).
\]
Thus
\[
      \Phi(d)-\Phi(\hat d)
      =  \bigl(\langle p, d\rangle-\langle p, \hat d\rangle\bigr)
         +\sum_{j\in V(d)}c_j r_j(d).
\]
The second term is bounded below by
$c_{\min}\sum_{j\in V(d)}  r_j(d) \geqslant c_{\min}\delta$,
so using the Lipschitz bound,
\[
      \Phi(d)-\Phi(\hat d)
      \geqslant \bigl(c_{\min}-\tfrac{\tilde{L}}{b_{\min}}\bigr)\delta.
\]
By definition, the coefficient of $\delta$ is positive,
hence $\Phi(d)>\Phi(\hat d)$ for $\hat{d}\in\mathcal{D}$, contradicting the optimality of $d$.
Therefore all global minimizers of \eqref{eq:pen} must lie in $F$.

On $F$ the penalty term vanishes, i.e., $\Phi(d)=\langle p, d\rangle+\sum_jc_j b_j$.
Thus \eqref{eq:UFD} and \eqref{eq:pen} share the same minimizers and
their optimal values differ only by the constant $\sum_jc_j b_j$.

\end{proof}

A specific choice of $c_j$ that satisfies \eqref{eq:lambda-cond-main} is
\begin{equation}\label{eq:cj}
    c_j=\frac{\|\nabla f_x(y)\|_2}{\epsilon-\frac{1}{2}Mg_{x,j}(y)}, \quad  j = 1,\cdots, l.
\end{equation}
The constant $\epsilon>0$ is a small positive number added to ensure numerical stability during the calculations. Intuitively, when $g_{x,j}(y)$ is close to zero, which indicates that the point $y$ lies near the boundary of the $j$-th inequality constraint, the corresponding weight $c_j$ should be larger, as such constraints are more likely to be violated in subsequent updates. By assigning higher weights to these near-active constraints, the network is encouraged to prioritize directions  $d$ that satisfy $\langle d,\nabla g_{x,j}(y)\rangle\leqslant -M g_{x,j}(y)$, which helps prevent constraint violations.   We remark there are  many possible surrogates for $c_j$, e.g., $c_j=\exp(-\delta g_{x,j}(y))$ or the softmax function.

\begin{proposition}\label{proj}
Let $H\in\mathbb{R}^{n\times m}$ denote the matrix formed by the gradients of the equality constraints
\begin{equation*}
    H=\begin{bmatrix}
        \nabla h_{x,1}(y)&\cdots&\nabla h_{x,m}(y)
    \end{bmatrix}.
\end{equation*}
Then, under Assumption~ \ref{assump:linear_indep_eq}, the expression for the projection onto $\mathcal{D}$ is given by
\begin{equation}\label{2}
    \mathcal{P}(d)=\left\{\begin{matrix}
        \hat{d},& ~ \text{if} ~ \| \hat{d}\|_2\leqslant 1,\\
        \hat{d}/\|\hat{d}\|_2,& ~ \text{otherwise},
    \end{matrix}\right. \quad\text{where }\quad \hat{d}=d-H(H^\top H)^{-1}H^\top d.
\end{equation}
\end{proposition}

\begin{proof}
Given any vector \( d \in \mathbb{R}^n \), we aim to compute its projection onto \( \mathcal{D} \), i.e., solve the following problem:
\begin{equation*}
\min_{d' \in \mathbb{R}^n} \frac{1}{2} \|d' - d\|_2^2 \quad \text{s.t.} \quad \|d'\|_2 \leqslant 1,\quad H^\top d'=0.
\end{equation*}

Without loss of generality, we assume that the Linear Independence Constraint Qualification (LICQ) holds. Otherwise, the projection reduces to the origin $d'=0$. Now we derive the KKT conditions for this problem from the Lagrangian
\[
\mathcal{L}(d', \lambda, \mu) = \frac{1}{2} \|d' - d\|_2^2 + \lambda^\top H^\top d' + \mu(\|d'\|_2^2 - 1),
\]
where \( \lambda \in \mathbb{R}^{n-m} \) and \( \mu \geqslant 0 \) are the Lagrange multipliers.

Taking the gradient with respect to \( d' \) and setting it to zero gives:
\begin{equation*}
    d' - d + H \lambda + 2\mu d' = 0 \quad \Rightarrow \quad (1 + 2\mu)d' + H \lambda = d.
\end{equation*}
Since $H^\top d' = 0$, we have
\begin{equation*}
    H^\top H \lambda=(1 + 2\mu)H^\top d' +H^\top H \lambda = H^\top d\quad\Rightarrow\quad
    \lambda=(H^\top H)^{-1} H^\top d.
\end{equation*}
We consider two cases:
\textbf{Case 1:} If \( \mu = 0 \), then the projection is
\begin{equation*}
    d'=d-H\lambda=d - H (H^\top H)^{-1} H^\top d=\hat{d}.
\end{equation*}
\textbf{Case 2:} If \( \mu > 0 \), then we have $\|d'\|=1$ and
\begin{equation*}
(1+2\mu)^2
= (1+2\mu)^2 (d')^\top d'
= (d - H\lambda)^\top (d - H\lambda)
= \hat{d}^\top \hat{d}
= \|\hat{d}\|^2.
\end{equation*}
Hence, the projection is:
\begin{equation*}
    d'=\frac{1}{1+2\mu}\left(d - H (H^\top H)^{-1} H^\top d\right)=\frac{1}{\|\hat{d}\|}\hat{d}.
\end{equation*}
\end{proof}

Consequently, the procedure of the projected subgradient method for solving this problem is as follows:
\begin{equation}\label{pgm}
d_{k+1}=\mathcal{P}\big(d_k-\gamma_k \mathbf{u}_k\big),
\end{equation}
where $\gamma_k>0$ is the step size and $\mathcal{P}$ is the projection operator defined in~\eqref{2}. Let $\textbf{1}_{\{\cdot\}}$ denotes the indicator function, the subgradient term $\mathbf{u}_k$ is given by
\begin{equation}\label{uk}
\mathbf{u} _k = \nabla f_x(y)+\sum_{j=1}^lc_j\textbf{1}_{\{\langle d^k,\nabla g_{x,j}(y)\rangle\geqslant -M g_{x,j}(y)\}}\nabla g_{x,j}(y).
\end{equation}

Note that in many practical problems, the matrix $H$ is fixed across instances or does not change frequently. In such cases, the projection in \eqref{2} can be precomputed, making the computational cost manageable. Examples include decision-focused learning setups, where the equality constraints remain constant across instances \cite{tan2020learning}. In other problems, the equality constraints are simple, allowing the projection to be computed efficiently; for instance, in classical portfolio optimization\cite{fabozzi2008portfolio}, the budget constraint enables a straightforward projection.

\subsection{Design of Descent Module}
Directly solving problem \eqref{eq:subproblem-1} using the projected subgradient method usually results in slow convergence, due to the use of diminishing step size. To address this, we propose Descent-Module, which is designed by unrolling the projected subgradient algorithm. In our proposed network architecture, each layer takes the form of one iteration of the projected (sub)gradient method,  
\begin{equation}
    d_{k+1}=\mathcal{P}\big(d_k- \gamma_kT^k(\mathbf{u}_k)\big).
\end{equation}
The key difference is that we apply learnable modules $T^k$ to the subgradient term $\mathbf{u}_k$, and the definition of $T^k$ is given as follows:
\begin{equation}\label{7}
    T^k(\mathbf{u}_k) = \mathbf{V}^k \text{ReLU}\left(\mathbf{W}^k \mathbf{u}_k + \mathbf{b}_1^k\right)+\mathbf{b}_2^k,
\end{equation}
where $\mathbf{W}^k\in\mathbb{R}^{q\times n},\mathbf{V}^k\in\mathbb{R}^{n\times q}$ and $\mathbf{b}_1^k\in\mathbb{R}^q,\mathbf{b}_2^k\in\mathbb{R}^n$ are the weight matrix and bias that we need to learn and $\text{ReLU}(x)=\max(x,0)$. The design of the operator $T$ follows the work in \cite{wu2024designing}, 
which demonstrates that such an architecture has strong universal approximation capabilities.

The step size $\gamma_k$ is also set as a learnable parameter, which avoids the need for manual tuning, and we provide in Appendix~\ref{exp:gamma} the learned values of $\gamma_k$ across different layers in our experiments.

The architectures of the Descent Module are illustrated in Figure~\ref{descent-module}. Each Descent Module consists of $K$ Descent Layers sharing the same architecture and the input to the first layer is chosen as $d_0 = -\nabla f_x(y)$. 
\begin{figure}[H]
  \centering
  \includegraphics[width=0.9\linewidth]{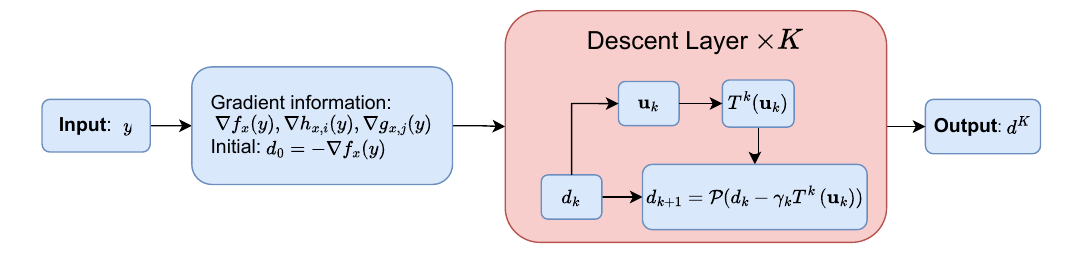} 
  \caption{Overall structure of the Descent Module}
  \label{descent-module}
\end{figure}

\begin{theorem}\label{th}
Let $d^*$ be the optimal solution of Problem~(\ref{eq:subproblem-1}), and let $g$ denote its objective function. For any $\varepsilon > 0$, there exists a $K_\varepsilon$-layer Descent-Module with a specific parameter assignment independent of $x$, whose output $d$ satisfies $\lvert g(d) - g(d^*)\rvert < \varepsilon$. Moreover, the number of layers satisfies $ K_\varepsilon \leqslant \frac{C}{\varepsilon^2}$ for some constant $C > 0$. 
\end{theorem}

The proof of Theorem~\ref{th} relies on two technical lemmas, which we present below before giving the full argument.

\begin{lemma}\label{lm1}
For any $\varepsilon > 0$, there exists a constant $C > 0$ such that if we set $K = \frac{C}{\varepsilon^2}$ and choose the step size in (\ref{pgm}) as $\gamma_k = \frac{1}{\sqrt{K}}$, then
\begin{equation*}
\min_{1 \leqslant k \leqslant K} g(d_k) - g(d^*) \leqslant \varepsilon,
\end{equation*}
where $d^*$ denotes the optimal solution of Problem~(\ref{eq:subproblem-1}) and $g$ denote its objective function.
\end{lemma}
This lemma is a direct consequence of the classical convergence theory of the projected subgradient method, and we therefore omit its proof.

\begin{lemma}\label{lm2}
    Given the sequence of iterates $\{d_1^{\text{proj}},\cdots,d_K^{\text{proj}}\}$ generated by the projected gradient method (\ref{pgm}) with initial input $d_0$, there exists a $K$-layer Descent-Net with a specific parameter assignment that, starting from the same initial input $d_0$, it produces the same iterative sequence, i.e., $d_k = d_k^{\text{proj}}$ for all $1 \leqslant k \leqslant K$.
\end{lemma}
\begin{proof}
It suffices to show that there exists a set of parameters such that \( T^k(\mathbf{u}_k) = \gamma_k \mathbf{u}_k \) for all $1\leqslant k\leqslant K$, where \( \mathbf{u}_k \) is defined in (\ref{uk}).

Let \( \mathbf{W}^k \in \mathbb{R}^{q \times n} \) be a full column rank matrix , so its left pseudo-inverse \( (\mathbf{W}^k)^\dagger \in \mathbb{R}^{n \times q} \) exists and satisfies \( (\mathbf{W}^k)^\dagger \mathbf{W}^k = I_n \).  

By assuption (\ref{assump:Uniform_boundedness}) and the defination of $\mathbf{u_k}$, we have 
\begin{equation*}
    \begin{aligned}
        \|\mathbf{u}_k\|_2&\leqslant \|\nabla f_x(y)\|_2+\sum_{j=1}^l \lvert c_j\rvert\cdot \sqrt{n}\cdot \lvert\nabla g_{x,j}(y)\rvert \\
        & \leqslant \|\nabla f_x(y)\|_2+\max_j(c_j)\cdot \sqrt{n}\cdot \|\nabla g_x(y)\|_1 \\
        & \leqslant \|\nabla f_x(y)\|_2+\max_j(c_j)\cdot n\cdot \|\nabla g_x(y)\|_2 \\
        & \leqslant L_f+\max_j(c_j)\cdot n\cdot L_g
    \end{aligned}
\end{equation*}
In the derivation of the third inequality, we used the equivalent norm theorem. Then we have 
\[
\|\mathbf{W}^k\mathbf{u}_k\|_1\leqslant\sqrt{n}\|\mathbf{W}^k\mathbf{u}_k\|_2\leqslant\sqrt{n}\|\mathbf{W}^k\|_2\|\mathbf{u}_k\|_2\leqslant \sqrt{n}\|\mathbf{W}^k\|_2(L_f+\max_j(c_j)\cdot n\cdot L_g)
\]
Let $L=\sqrt{n}\|\mathbf{W}^k\|_2(L_f+\max_j(c_j)\cdot n\cdot L_g)$. Define the bias vector as \( \mathbf{b}_1^k = L \cdot \mathbf{1}_q \), where \( \mathbf{1}_q \) denotes the \( q \)-dimensional vector with all entries equal to one. Then we have

\[
\text{ReLU}(\mathbf{W}^k \mathbf{u}_k + \mathbf{b}_1^k) = \mathbf{W}^k \mathbf{u}_k + \mathbf{b}_1^k,
\]
since each coordinate of \( \mathbf{W}^k \mathbf{u}_k + \mathbf{b}_1^k \) is positive.

Now, let the second layer weight be \( \mathbf{V}^k = \gamma_k (\mathbf{W}^k)^\dagger \), and the second bias be \( \mathbf{b}_2^k = -\gamma_k (\mathbf{W}^k)^\dagger \mathbf{b}_1^k \). Then we have:
\[
T^k(\mathbf{u}_k) = \mathbf{V}^k \cdot \text{ReLU}(\mathbf{W}^k \mathbf{u}_k + \mathbf{b}_1^k) + \mathbf{b}_2^k = \gamma_k (\mathbf{W}^k)^\dagger (\mathbf{W}^k \mathbf{u}_k + \mathbf{b}_1^k) - \gamma_k (\mathbf{W}^k)^\dagger \mathbf{b}_1^k = \gamma_k \mathbf{u}_k.
\]
This completes the proof.
\end{proof}

We now proceed to prove Theorem~\ref{th}. 
\begin{proof}
First, by Lemma~\ref{lm1}, we know that for any $\varepsilon > 0$, by choosing an appropriate step size, there exists an iteration sequence $\{d_k^{\text{proj}}\}_{k=1}^K$ generated by the projected subgradient method such that
\begin{equation*}
    \lvert \min_{1\leqslant k\leqslant K} g(d^{\text{proj}}_k) - g(d^*)\rvert < \varepsilon.
\end{equation*}
Let $K_\varepsilon=\text{argmin}_{1\le k\leqslant K}\:g(d^{\text{proj}}_k)$, then we have $K_\varepsilon\leqslant K=C/\varepsilon^2$ and 
\begin{equation*}
    \lvert g(d^{\text{proj}}_{K_\varepsilon}) - g(d^*)\rvert < \varepsilon.
\end{equation*}
Moreover, by Lemma \ref{lm2}, we know that there exists there exists a $K_\varepsilon$ layer Descent-Net such that the output of each layer exactly matches the corresponding iterate sequence $\{d_k^{\text{proj}}\}_{k=1}^{K_\varepsilon}$. In particular, we have
\begin{equation*}
    d_{K_\varepsilon} = d^{\text{proj}}_{K_\varepsilon}
\end{equation*}
Therefore,
\begin{equation*}
    \lvert g(d_{K_\varepsilon}) - g(d^*) \rvert < \varepsilon,
\end{equation*}
which is exactly the desired result.
\end{proof}

\subsection{Step size}

After obtaining the descent direction $d$ from the Descent module, we still need to determine a suitable step size. We assume that all constraints are linear. Since the Descent module contains the projection operator $\mathcal{P}$, the final descent direction $d$ produced by Descent module is orthogonal to the gradients of the equality constraints. Therefore, updating along $d$ will not violate the equality constraints. 

We only need to ensure that the step size is not too large to violate the inequality constraints. For each linear inequality constraint $g_{x,j}$, we have:
\begin{equation*}
    g_{x,j}(y+\alpha d)= g_{x,j}(y)+\alpha\cdot\langle d,\nabla g_{x,j}(y)\rangle.
\end{equation*}
If $\langle d,\nabla g_{x,j}(y)\rangle>0$, updating the solution along $d$ will increase $g_{x,j}$. To preserve the feasibility of the inequality constraint, i.e., $g_{x,j}(y+\alpha d)\leqslant 0$, the step size $\alpha$ must satisfy
\begin{equation*}
    \alpha\leqslant \frac{-g_{x,j}(y)}{\langle d,\nabla g_{x,j}(y)\rangle}.
\end{equation*}
Therefore, the maximum allowable step size is given by
\begin{equation}\label{eq:alpha}
    \alpha_{\max}=\min_{j\in\mathcal{I}}\frac{-g_{x,j}(y)}{\langle d,\nabla g_{x,j}(y)\rangle},\quad\text{where }~\mathcal{I}=\{j\mid \langle d,\nabla g_{x,j}(y)\rangle>0\}.
\end{equation}

To guarantee a sufficient decrease of the objective value, the step size $\alpha$ should also satisfy $f_x(y+\alpha d)<f_x(y)$. To obtain a sufficient decrease in the objective value, we introduce a learnable parameter $\beta \in \mathbb{R}$ and use the sigmoid function $\sigma(\cdot)$ to map it into $(0,1)$. We then use this factor to scale $\alpha_{\max}$, and the final update rule for $y$ is 
\begin{equation}\label{eq:update rule}
    y^{\mathrm{new}}=y^{\mathrm{old}}+\sigma(\beta)\alpha_{\max} \cdot d.
\end{equation}

In addition, if the descent direction $d$ obtained from the Descent-Net is the optimal solution of Problem~(\ref{eq:subproblem-1}), Lemma~\ref{lem:exact_penalty_l2} ensures that a fixed step size of $\alpha = 1/M$ is feasible. However, we found that such a fixed step size does not perform well in practice, and in the appendix \ref{exp:step size} we provide a comparison of different step-size selection strategies.

\subsection{Descent-Net}

Our Descent-Net consists of $S$ Descent Modules. The input of the network is an initial feasible solution $y_0$ of Problem~\eqref{prob:constraint}. At each stage, the $s$-th module takes the gradient information at the current iterate $y_s$ and outputs a descent direction $d_s$, which is then used to update the solution to $y_{s+1}$ according to the update rule~\eqref{eq:update rule}. By repeatedly updating the solution in this manner, the network finally produces a high-accuracy feasible solution $y_S$. The overall procedure of the proposed method is summarized in Algorithm~\ref{alg:training}.

\begin{algorithm}[H]
\caption{Descent-Net}
\label{alg:training}
\begin{algorithmic}[1]
\State \textbf{Input:} initial feasible point $y_0\in\mathcal{C}$, $S$ modules and $K$ layers in each module.
\State \textbf{Learnable parameters:} $\Theta := \{\mathbf{V}^k, \mathbf{W}^k, \mathbf{b}_1^k,\mathbf{b}_2^k,\gamma_k\}_{k=0,1,\cdots, K-1}$ and $\{\beta_s\}_{s=0,\cdots, S-1}.$
\For{$s=0, 1,\cdots, S-1$}
\State $d_0=-\nabla f_x(y_s)$
\For{$k =0, \cdots, K-1$}
    \State $\mathbf{u}_k=\nabla f_x(y_s) +\sum_{j=1}^lc_j\textbf{1}_{\{\langle d_k,\nabla g_j(y_s)\rangle\geqslant -M g_{x,j}(y)\}}\nabla g_j(y_s)$
    \State $d_{k+1}=\mathcal{P}\left(d_k-\gamma_kT^k(\mathbf{u}_k)\right)$, where $\mathcal{P}$ is defined in (\ref{2}) and $T^k$ is defined in (\ref{7}) with parameters $\{\mathbf{V}^k, \mathbf{W}^k, \mathbf{b}_1^k,\mathbf{b}_2^k\}$
\EndFor
\State $y_{s+1}= y_s+ \sigma(\beta_s)\cdot \alpha_s d_K$ as defined by \eqref{eq:update rule}, where $\alpha_s$ is obtained by \eqref{eq:alpha}.
\EndFor
\State Train the parameters with loss: $\ell_p(y) = f_x(y) + \lambda_g \| \text{ReLU}(g_x(y)) \|_1 + \lambda_h \| h_x(y) \|_1$
\State \textbf{Output:} $y_S.$
\end{algorithmic}
\end{algorithm}

\begin{figure}[H]
  \centering
  \includegraphics[width=0.9\linewidth]{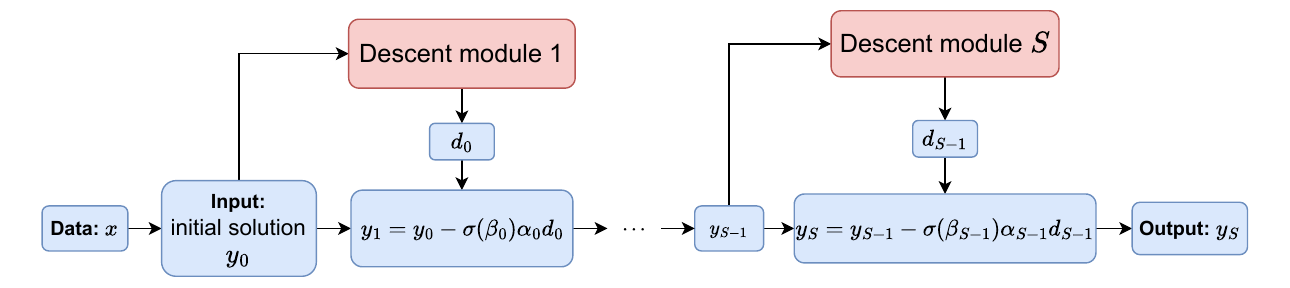} 
  \caption{Architecture of the entire network.}
  \label{descent-net}
\end{figure}

\begin{theorem}[global convergence of the Descent-Net]
\label{thm:desnet-conv}
Suppose the Assumptions ~\ref{assump:close-bounded},~\ref{assump:linear_indep_eq},~\ref{assump:licq} hold. In addition, assume that $h_x, g_x$ are linear. Then there exists $K_\varepsilon$-layer Descent-Module with a specific parameter assignment independent of $x$ and $S>0$ such that   the Descent-Net generates a  KKT   conditions of the   problem \eqref{prob:constraint}.
\end{theorem}

To prove Theorem~\ref{thm:desnet-conv}, we require several technical 
definitions and lemmas, which we present below before proceeding to the proof.

\begin{definition}[Fritz–John point]
Let $y \in \mathbb{R}^n$ be a feasible point for the problem
\[
\min f(y) \quad \mathrm{s.t.} \quad h_i(y)=0,\; g_j(y)\leqslant 0.
\]
Then $y$ is called a \emph{Fritz–John point} if there exist multipliers
$\lambda_0 \geqslant 0$, $\lambda_j \geqslant 0$ for $j=1,\cdots,l$,
and $\mu_i \in \mathbb{R}$ for $i=1,\cdots,m$, not all zero, such that
\begin{align*}
\lambda_0 \nabla f(y) + \sum_{j=1}^l \lambda_j \nabla g_j(y)
+ \sum_{i=1}^m \mu_i \nabla h_i(y) &= 0, \\
\lambda_j \cdot g_j(y) &= 0,\quad j=1,\cdots,l.
\end{align*}
\end{definition}

\begin{definition}[KKT point]
A feasible point $y \in \mathbb{R}^n$ is called a \emph{Karush–Kuhn–Tucker (KKT) point}
if there exist multipliers $\lambda_j \geqslant 0$ and $\mu_i \in \mathbb{R}$ such that
\begin{align*}
\nabla f(y) + \sum_{j=1}^l \lambda_j \nabla g_j(y)
+ \sum_{i=1}^m \mu_i \nabla h_i(y) &= 0, \\
\lambda_j \cdot g_j(y) &= 0,\quad j=1,\cdots,l.
\end{align*}
\end{definition}
If the LICQ condition holds at $\bar{y}$, then the Fritz-John point is also a KKT point.

\begin{lemma}[Farkas Lemma with equality constraints]
\label{lem:farkas-equality}
Let \(A \in \mathbb{R}^{m \times n}\), \(B \in \mathbb{R}^{p \times n}\),
and \(b \in \mathbb{R}^m\).

Then exactly one of the following two systems has a solution:

\begin{enumerate}[label=(\alph*)]
\item There exists \(x \in \mathbb{R}^n\) such that
\[
A x < b, \quad B x = 0.
\]

\item There exists \((\lambda, \mu) \in \mathbb{R}^m \times \mathbb{R}^p\), not both zero, such that
\[
\lambda \geqslant 0,\quad
A^\top \lambda + B^\top \mu = 0,\quad
\lambda^\top b \leqslant 0.
\]
\end{enumerate}
Moreover, both systems cannot be simultaneously feasible.
\end{lemma}

We now show that the problem \eqref{eq:UFD} has negative value if $y$ is not a Fritz-John point. 
\begin{lemma}[Descent direction under failure of Fritz--John]
\label{lem:farkas-descent}
Let $\bar y \in \mathcal{C}$ be a feasible point. Suppose that the set of vectors
$$
\left\{ v = \lambda_0 \nabla f_x(\bar y) + \sum_{j=1}^l \lambda_j \nabla g_j(\bar y) 
+ \sum_{i=1}^m \mu_i \nabla h_i(\bar y)
\:\middle\vert\:
\lambda_0 \geqslant 0,\ \lambda_j \geqslant 0,\ (\lambda,\mu) \ne 0,\ \lambda_j g_j(\bar y) = 0 \right\}
$$
does not contain the zero vector. That is, $\bar y$ is not a Fritz--John point.

Then there exists a vector $d \in \mathbb{R}^n$ such that:
\begin{itemize}
    \item \(\nabla h_x(\bar y)^\top d = 0\),
    \item \(\nabla f_x(\bar y)^\top d < 0\),
    \item \(\nabla g_j(\bar y)^\top d < -M g_j(\bar y)\) for all $j=1,\cdots,l$,
    \item \(\|d\|_2 \leqslant 1\).
\end{itemize}
\end{lemma}

\begin{proof}
Let \(H := [\nabla h_1(\bar y), \cdots, \nabla h_m(\bar y)] \in \mathbb{R}^{n \times m}\), and define the subspace of directions satisfying the linearised equality constraints:
\[
\mathcal{T} := \left\{ d \in \mathbb{R}^n \mid H^\top d = 0 \right\}.
\]

Let \(A \in \mathbb{R}^{(l+1) \times n}\) be the matrix whose rows are:
\[
A := \begin{bmatrix}
\nabla f_x(\bar y)^\top \\
\nabla g_1(\bar y)^\top \\
\vdots \\
\nabla g_l(\bar y)^\top
\end{bmatrix},
\qquad
b := \begin{bmatrix}
0 \\
- M g_1(\bar y) \\
\vdots \\
- M g_l(\bar y)
\end{bmatrix}.
\]

Then we consider the system:
\[
A d < b, \quad \text{subject to } H^\top d = 0.
\]

Since \(\bar y\) is not a Fritz–John point, the system of equalities
\[
\lambda_0 \nabla f_x(\bar y)
+ \sum_j \lambda_j \nabla g_j(\bar y)
+ \sum_i \mu_i \nabla h_i(\bar y) = 0
\quad\text{with } \lambda_0\geqslant0,\ \lambda_j\geqslant0, \mu_i\ \text{not all zero},
\]
has no solution satisfying the complementarity condition
\(\lambda_j g_j(\bar y)=0\).

Therefore, by the  Farkas Lemma \ref{lem:farkas-equality}, the dual system:
\[
\text{find } d \in \mathcal{T} \text{ such that } A d < b
\]
is feasible.

Because \(A, b\) are fixed and \(b_j = - M g_j(\bar y)\geqslant 0\), which is finite for all \(j=1,\cdots,l\), and \(\mathcal{T}\) is a linear subspace, the feasible set is convex and open in \(\mathcal{T}\). We can scale \(d\) such that \(\|d\|_2 \leqslant 1\).

Hence, such a direction \(d\) exists satisfying the claimed conditions.
\end{proof}

By Lemma~\eqref{lem:exact_penalty_l2} and Theorem~\ref{th}, to establish 
Theorem~\ref{thm:desnet-conv} it suffices to show that if the following 
algorithm—constructed from the subproblem~\eqref{eq:pen}—converges to a KKT 
point. We therefore begin by analyzing the convergence of the algorithm stated below.

\paragraph{Algorithm (UFD–penalty method).}
Given a feasible starting point $y_{0}\in\mathcal{C}$, repeat for $k=0,1,\cdots$
\begin{enumerate}[label=\arabic*.]
\item
 With the condition \eqref{eq:lambda-cond-main} holds, solve the sub–problem \eqref{eq:pen} 
  at the current iterate $y_k$ and obtain a minimizer $d_k$.
\item
  Update \quad $y_{k+1}\;=\;y_k+\alpha_k d_k$, where    
  $\alpha_k:=\arg\min_{\alpha}\{ f_x(y_k+\alpha d_k)\mid \alpha \in (0,1/M]\}$,
      \hfill(\emph{Lemma \ref{lem:exact_penalty_l2} implies
      $y_{k+1}\in\mathcal{C}$})
\end{enumerate}
We have the following result, which is similar to the  Topkis--Veinott method \cite{zoutendijk1960methods,faigle2013algorithmic}.
\begin{theorem}[global convergence of the UFD–$L_{2}$ method]
\label{thm:ufd-conv}
Suppose the Assumption \ref{assump:close-bounded}~\ref{assump:linear_indep_eq} and \ref{assump:licq} hold. Furthermore, we assume
that the gradient $\nabla f_x(y)$ is $L-$Lipschitz continuous, and $h_x, g_x$ are linear. Then every accumulation point $\bar y$ of the sequence
$\{y_k\}$ generated by the UFD-penalty  algorithm
satisfies the KKT   conditions of the   problem \eqref{prob:constraint}.
\end{theorem}

\begin{proof}
\textbf{(i) Feasibility and boundedness.}
The proof of Lemma~\ref{lem:exact_penalty_l2} shows that every $d_k$ satisfies
\(
\nabla g_j(y_k)^{\!\top}d_k\leqslant -M g_j(y_k)
\),
hence
\(
g_j(y_{k+1})=g_j(y_k)+\alpha\nabla g_j(y_k)^\top d_k\leqslant0,
\)
where $\alpha \in(0,1/M].$
Equality constraints are preserved by
$\nabla h_x(y_k)^\top d_k=0$, so $y_{k+1}\in\mathcal{C}$.
Because $\{y\in\mathcal{C}\mid f_x(y)\leqslant f_x(y_0)\}$ is bounded, $\{y_k\}$ is bounded and admits
convergent subsequences.
 
 
\noindent\textbf{(ii) every accumulation point is a Fritz--John point.}

Let $\bar y$ be an accumulation point of $\{y_k\}$, extracted from a subsequence $\{y_{k_s}\}$. Suppose, for contradiction, that $\bar y$ is not a Fritz--John point.

Then, by Lemma~\ref{lem:farkas-descent}, there exist $z < 0$ and a direction $d \in \mathbb{R}^n$ satisfying:
\[
\|d\|_2 \leqslant 1, \quad \nabla h_x(\bar y)^\top d = 0,
\quad \nabla f_x(\bar y)^\top d < z < 0,
\quad \nabla g_j(\bar y)^\top d < -M g_j(\bar y) + z.
\]

Since $f_x$, $g_j$, and all gradients are continuous, and \(y_{k_s} \to \bar y\), there exists $\varepsilon > 0$ and $\delta > 0$ such that for all $s$ sufficiently large (i.e., \(\|y_{k_s} - \bar y\| < \delta\)):
\[
\begin{aligned}
\nabla f_x(y_{k_s})^\top d &< z + \varepsilon, \\
\nabla g_j(y_{k_s})^\top d &< -M g_j(y_{k_s}) + \varepsilon, \\
\nabla h_x(y_{k_s})^\top d &< \varepsilon.
\end{aligned}
\]

Fix \(\varepsilon := \lvert z\rvert/3 > 0\). Then for large \(s\), we obtain:
\[
\begin{aligned}
\nabla f_x(y_{k_s})^\top d &< z + \varepsilon =: \hat z < 0, \\
\nabla g_j(y_{k_s})^\top d &< -M g_j(y_{k_s}) + \varepsilon, \\
\nabla h_x(y_{k_s})^\top d &< \varepsilon.
\end{aligned}
\]

Now consider the solution \(d_{k_s}\) of the UFD subproblem \eqref{eq:Topkis-2} at \(y_{k_s}\), which satisfies:
\[
\|d_{k_s}\|_2 \leqslant 1, \quad
\nabla h_x(y_{k_s})^\top d_{k_s} = 0, \quad
\nabla g_j(y_{k_s})^\top d_{k_s} \leqslant -M g_j(y_{k_s}).
\]

Since \(d\) is a feasible direction and \(\nabla f_x(y_{k_s})^\top d < \hat z < 0\), it follows that the optimal value \(z_s := \nabla f_x(y_{k_s})^\top d_{k_s}\) must also satisfy:
\[
z_s < \hat z < 0.
\]

Thus, for all large \(s\), we have:
\[
\nabla f_x(y_{k_s})^\top d_{k_s} = z_s < 0, \quad
\nabla g_j(y_{k_s})^\top d_{k_s} < 0, \quad
\nabla h_x(y_{k_s})^\top d_{k_s} = 0.
\]

Now define \(y_{k_{s+1}} := y_{k_s} + t d_{k_s}\), where \(t > 0\) is small. Since \(d_{k_s}\) satisfies the linearized equality constraints exactly and inequality constraints strictly, Taylor expansion gives:
\[
\begin{aligned}
f_x(y_{k_s} + t d_{k_s}) &= f_x(y_{k_s}) + t \nabla f_x(y_{k_s})^\top d_{k_s} + o(t)
< f_x(y_{k_s}) + t z_s/2, \\
g_j(y_{k_s} + t d_{k_s}) &= g_j(y_{k_s}) + t \nabla g_j(y_{k_s})^\top d_{k_s} + o(t) < 0, \\
h_i(y_{k_s} + t d_{k_s}) &= h_i(y_{k_s}) + t \nabla h_i(y_{k_s})^\top d_{k_s} = 0.
\end{aligned}
\]

Therefore, for sufficiently small \(t > 0\), the updated point \(y_{k_{s+1}} := y_{k_s} + t d_{k_s}\) remains feasible and decreases the objective value. 

This contradicts the assumption that $\{f_x(y_k)\}$ converges to a finite value (since it would go to \(-\infty\)). Hence, our assumption must be false: every limit point $\bar y$ must satisfy the Fritz--John condition.

\medskip
\textbf{(iii) LICQ $\Rightarrow$ KKT.}
Under LICQ the Fritz–John multipliers have $\lambda_0>0$,
so the KKT system holds at $\bar y$.
\end{proof}

\section{Experiment}

We evaluate our Descent-Net on three types of problems: convex quadratic programs, a simple class of non-convex optimization problems, and the AC optimal power flow (AC-OPF) problem. To further evaluate the scalability of our method and to demonstrate its performance on a practical real-world instance of quadratic programming, we additionally include a large-scale portfolio optimization task. Detailed experimental settings are provided in Appendix~\ref{exp_details}.

\subsection{Baselines and Evaluation Criteria}

We compare our method against several benchmarks, including: 
\begin{itemize}
    \item \textbf{Optimizer}: Traditional solvers include \verb|OSQP| \cite{stellato2020osqp} and \verb|qpth| \cite{amos2017optnet} for convex QPs; the GPU-accelerated \verb|HPR-QP| \cite{chen2025hpr} and \verb|CuClarabel| \cite{CuClarabel}, both of which are executed on an H100 GPU in our experiments; \verb|IPOPT| \cite{wachter2006implementation} for general non-convex problems; and the \verb|PYPOWER| solver, a Python port of \verb|MATPOWER| \cite{zimmerman2005matlab}, for AC-OPF.
    \item \textbf{DC3}\cite{dontidc3}: The full DC3 framework that combines both completion and correction operators. 
    \item \textbf{Projection method}:Trains an MLP and projects its output onto the feasible set. For problems with linear constraints (convex QP and simple non-convex cases), the projection is solved using OptNet~\cite{amos2017optnet}. For the AC-OPF problem, the projection follows the differentiable solver of \cite{chen2021enforcing}.
    \item \textbf{Warm start}: The infeasible NN prediction is directly used as the warm-start for the optimizer of \cite{chen2021enforcing}, following the warm-starting schemes of \cite{diehl2019warm} and \cite{baker2019learning}.
    \item \textbf{CBWF}\cite{wu2025constraint}: Inspired by the classical active set method, this approach explores the boundaries around inequality constraints and updates the initial solution to obtain a better objective value.
\end{itemize}
Other approaches reported in the literature, such as simple neural network models trained with penalty functions, are not included due to poor constraint satisfaction.

The performance of all methods is assessed according to the following criteria:  
\begin{itemize}
    \item \textbf{Feasibility:} 
    measured by the average constraint violation of both equality and inequality constraints, i.e., $\frac{1}{m}\sum_{i=1}^{m} \lvert h_{x,i}(y) \rvert$ and $\frac{1}{l}\sum_{j=1}^{l} \mathrm{ReLU}\left(g_{x,j}(y)\right)$. 
    \item \textbf{Optimality:} measured by the average relative and absolute errors (in the $\ell_1$ norm) for both the solution and the objective value, where the optimal solution is approximated by optimizer.  
    \item \textbf{Efficiency:} the computational time. It is worth noting that \verb|OSQP|, \verb|HPR-QP|, \verb|CuClarabel|, \verb|IPOPT|, and \verb|PYPOWER| only support sequential solving. For these solvers, we report the average runtime per instance to approximate full parallelization, while for the other methods the runtime is measured as the average over batches, where each batch of test instances is solved in parallel. And for CBWF and Descent-Net, the reported runtime includes both the time to obtain the initial solution and the time spent on refining the solution.
\end{itemize}

\subsection{Convex quadratic programs}

We first consider convex QPs with quadratic objectives and linear constraints:
\begin{equation}\label{prob:qp}
    \min_{y \in \mathbb{R}^n}\:\frac{1}{2} y^T Q y + p^T y, \quad \text{s.t. } Ay = x, \; Gy \leqslant h,
\end{equation}
where $Q \in \mathbb{R}^{n \times n} \succeq 0$, $p \in \mathbb{R}^n$, $A \in \mathbb{R}^{n_{eq} \times n}$, $G \in \mathbb{R}^{n_{ineq} \times n}$, and $h \in \mathbb{R}^{n_{ineq}}$ are fixed. The input $x \in \mathbb{R}^{n_{eq}}$ varies across problem instances, and the goal is to approximate the optimal $y$ given $x$. 

We generated 10,000 instances of $x$ for three problem sizes, $n=100$, $n=1000$, and $n=5000$. For all settings, Descent-Net uses the solutions produced by the DC3 model as its initial points. For $n=100$, the experimental results are summarized in Table~\ref{tab:qp100}, and the final solutions obtained by Descent-Net attain a relative objective error of $3.1\times 10^{-4}$. Note that the runtime reported for \verb|OSQP| corresponds to the average time per instance, as it only supports sequential solving and is therefore less efficient than Descent-Net.

\begin{table}[htbp]
\caption{Results on the convex QP task evaluated on the test set with 833 samples, 
with problem dimension $n=100$, $n_{eq}=50$, and $n_{ineq}=50$. For methods that support parallel execution, the runtime corresponds to the time required to process one batch with a batch size of 833.}
\label{tab:qp100}
\centering
\begin{tabular}{lccccc}
\toprule
Method & ineq. vio. & eq. vio. & sol. rel. err. & obj. rel. err. & Time (s) \\
\midrule
OSQP & 0.0000 & 0.0000 & 0.0  & 0.0 & 0.0028 \\
HPR-QP & 0.0000 & 0.0000 & $2.3\times10^{-8}$ & $6.3\times10^{-9}$ & 0.1080 \\
CuClarabel & 0.0000 & 0.0000 & $1.6\times10^{-3}$ & $1.8\times10^{-5}$ & 0.0539\\
qpth & 0.0000 & 0.0000 & $2.4\times10^{-6}$ & $4.7\times10^{-10}$ & 1.7593 \\
Projection method & 0.0000 & 0.0000 & $3.2\times10^{-2}$ & $8.4\times10^{-4}$ & 0.2124 \\
CBWF & 0.0000 & 0.0000 & $2.1\times10^{-1}$ & $6.6\times10^{-2}$ & 0.0366 \\
DC3 & 0.0000 & 0.0000 & $5.2\times10^{-1}$ & $4.2\times10^{-1}$ & 0.0036 \\
\textbf{Descent-Net} & 0.0000 & 0.0000 & $1.2\times10^{-2}$ & $3.1\times10^{-4}$ & 0.0132 \\
\bottomrule
\end{tabular}
\end{table}

For the larger scales $n=1000$ and $n=5000$, the corresponding results are presented in Table~\ref{tab:qp1000} and Table~\ref{tab:qp5000}, respectively. Descent-Net achieves relative objective errors of $5.8\times 10^{-3}$ for $n=1000$ and $3.2\times 10^{-3}$ for $n=5000$. The lower accuracy compared to the $n=100$ case can be attributed to the fact that the matrices have condition numbers on the order of $n^2$, causing the underlying solution map to become increasingly ill-conditioned as $n$ grows and thus making the approximation task more challenging for the neural network.

Nonetheless, the speed advantage of Descent-Net becomes more pronounced as the problem scale increases. For a fair comparison, we also include \verb|HPR-QP fast| and \verb|CuClarabel fast|, where each solver is terminated once it reaches the same accuracy level as Descent-Net (a relative objective error of about $10^{-3}$). Both methods require more time than Descent-Net to attain this accuracy. In conclusion, our Descent-Net obtains a solution for large-scale QP problems with accuracy $10^{-3}$ that is about 30 times faster than the baseline solvers.

\begin{table}[htbp]
\caption{Results on the convex QP task evaluated on the test set with 833 samples, 
with problem dimension $n=1000$, $n_{eq}=500$, and $n_{ineq}=500$. For methods that support parallel execution, the runtime corresponds to the time required to process one batch with a batch size of 32.}
\label{tab:qp1000}
\centering
\begin{tabular}{lccccc}
\toprule
Method & ineq. vio. & eq. vio. & sol. rel. err. & obj. rel. err. & Time (s) \\
\midrule
OSQP & 0.0000 & 0.0000 & 0.0  & 0.0 & 2.2174 \\
HPR-QP & 0.0000 & 0.0000 & $2.3\times10^{-6}$ & $1.8\times10^{-9}$ & 0.2710 \\
HPR-QP fast & 0.0008 & 0.0092 & $2.4\times10^{-3}$ & $1.1\times10^{-3}$ & 0.2191 \\
CuClarabel & 0.0000 & 0.0000 & $5.6\times10^{-4}$ & $3.0\times10^{-7}$ & 0.3822\\
CuClarabel fast & 0.0000 & 0.0000 & $3.8\times10^{-2}$ & $4.7\times10^{-3}$ & 0.2058 \\
qpth & 0.0000 & 0.0000 & $2.7\times10^{-8}$ & $9.7\times10^{-11}$ & 117.3808 \\
DC3 & 0.0000 & 0.0000 & $9.5\times10^{-1}$ & $1.1$ & 0.0036 \\
\textbf{Descent-Net} & 0.0000 & 0.0000 & $2.9\times10^{-2}$ & $5.8\times10^{-3}$ & 0.0078 \\
\bottomrule
\end{tabular}
\end{table}

\vspace{-1em}

\begin{table}[htbp]
\caption{Results on the convex QP task evaluated on the test set with 833 samples, 
with problem dimension $n=5000$, $n_{eq}=500$, and $n_{ineq}=500$. For methods that support parallel execution, the runtime corresponds to the time required to process one batch with a batch size of 8.}
\label{tab:qp5000}
\centering
\begin{tabular}{lccccc}
\toprule
Method & ineq. vio. & eq. vio. & sol. rel. err. & obj. rel. err. & Time (s) \\
\midrule
OSQP & 0.0000 & 0.0000 & 0.0  & 0.0 & 107.6874 \\
HPR-QP & 0.0000 & 0.0000 & $5.8\times10^{-8}$ & $1.1\times10^{-11}$ & 0.9160 \\
HPR-QP fast & 0.1545 & 0.3811 & $8.7\times10^{-3}$ & $1.1\times10^{-3}$ & 0.5493 \\
CuClarabel & 0.0000 & 0.0000 & $6.1\times10^{-6}$ & $2.1\times10^{-9}$ & 0.4681 \\
CuClarabel fast & 0.0000 & 0.0000 & $1.3\times10^{-2}$ & $1.4\times10^{-3}$ & 0.3072 \\
DC3 & 0.0000 & 0.0000 & $9.8\times10^{-1}$ & $9.2\times10^{-1}$ & 0.0038 \\
\textbf{Descent-Net} & 0.0000 & 0.0000 & $1.7\times10^{-2}$ & $3.2\times10^{-3}$ & 0.0153 \\
\bottomrule
\end{tabular}
\end{table}

In addition, to further illustrate the effectiveness of Descent-Module, we examine the error between the descent direction $d$ and the optimal solution of its corresponding subproblem~\eqref{eq:subproblem-1}. The experimental results are provided in appendix~\ref{exp:subprob}. We also compare Descent-Module with the original projected subgradient method, and the results are reported in appendix~\ref{exp:pgm}.

\subsection{Simple non-convex optimization}

We now examine a simple non-convex adaptation of the quadratic program 
\begin{equation}\label{prob:nonconvex}
    \min_{y \in \mathbb{R}^n}\:\frac{1}{2} y^T Q y + p^T \sin(y), \quad \text{s.t. } Ay = x, \; Gy \leqslant h,
\end{equation}
where $\sin(y)$ represents the component-wise application of the sine function to the vector $y$. Compared to problem~\eqref{prob:qp}, the only difference is that $y$ in the objective function is replaced with $\sin(y)$, which makes the problem non-convex. 

The experimental results are presented in Table~\ref{tab:nonconvex}. The initial solutions use those from DC3, and the final solutions produced by Descent-Net achieve a relative objective error of $2.3\times 10^{-4}$. Moreover, Descent-Net solves the instances approximately 19 times faster than the solver \verb|IPOPT|.

\begin{table}[htbp]
\caption{Results on the simple non-convex task evaluated on the test set with 833 samples, with problem dimension $n=100$, $n_{eq}=50$, and $n_{ineq}=50$. For methods that support parallel execution, the runtime corresponds to the time required to process one batch with a batch size of 833.}
\label{tab:nonconvex}
\centering
\begin{tabular}{lccccc}
\toprule
Method & ineq. vio. & eq. vio. & sol. rel. err. & obj. rel. err. & Time (s) \\
\midrule
IPOPT                 & 0.0000 & 0.0000 & 0.0 & 0.0 & 0.3364  \\
Projection method     & 0.0000 & 0.0000 & $5.4\times 10^{-2}$ & $1.8\times 10^{-3}$ & 0.2472 \\
CBWF                  & 0.0000 & 0.0000 & $2.6\times 10^{-1}$ & $5.5\times 10^{-2}$ & $0.0364$ \\
DC3                   & 0.0000 & 0.0000 & $4.4\times 10^{-1}$ & $3.1\times 10^{-1}$ & 0.0025 \\
\textbf{Descent-Net} & $0.0000$ & $0.0000$ & $1.5\times 10^{-2}$ & $2.3\times 10^{-4}$ & 0.0145 \\
\bottomrule
\end{tabular}
\end{table}

\subsection{Portfolio optimization}

A widely applicable instance of quadratic programming in real-world settings is the mean-variance portfolio optimization problem. The objective is to minimize portfolio risk while satisfying practical portfolio allocation constraints:
\begin{equation}
\min_{\mathbf{w}\in\mathbb{R}^n} \ \mathbf{w}^\top \Sigma \mathbf{w}
\quad \text{s.t.} \quad
\mathbf{w}^\top \mathbf{1} = 1,\;
\mathbf{w}^\top \mu \geqslant r_{\min},\;
\mathbf{w} \geqslant 0,
\label{eq:portfolio}
\end{equation}
where $\mathbf{w}$ denotes asset weights, $\Sigma\in\mathbb{R}^{n\times n}$ is the covariance matrix, $\mu\in\mathbb{R}^n$ is the expected return vector, and $r_{\min}$ is the minimum return requirement.

We conduct portfolio optimization experiments with $n=100$, $n=800$, and $n=4000$ assets to evaluate the practical effectiveness and scalability of our method. For each problem size, we generate 10,000 synthetic benchmark instances and use the equal-weighted portfolio $w_i = 1/n$ as the initial solution. Descent-Net is compared against the solver \verb|OSQP| and DC3, with the combined numerical results summarized in Table~\ref{tab:portfolio_all}.

\begin{table}[htbp]
\caption{Results on the portfolio optimization task evaluated on the test set with 1000 samples. 
The times of Descent-Net correspond to the average runtime for a batch of instances, with batch sizes of 512, 100, and 10 for problem sizes $n=100$, $n=800$, and $n=4000$, respectively.}
\label{tab:portfolio_all}
\centering
\begin{tabular}{lccccc}
\toprule
\multicolumn{6}{c}{\textbf{Number of assets $n=100$}} \\
\midrule
Method & ineq. vio. & eq. vio. & sol. rel. err. & obj. rel. err. & Time (s) \\
\midrule
OSQP        & $0.0000$ & $0.0000$ & $0.0$ & $0.0$ & $0.0015$ \\
HPR-QP      & $0.0000$ & $0.0000$ & $6.9 \times 10^{-6}$ & $1.1 \times 10^{-7}$ & $0.0645$ \\
CuClarabel     & $0.0000$ & $0.0000$ & $1.1 \times 10^{-5}$ & $4.0 \times 10^{-9}$ & $0.0407$ \\
DC3         & $0.0000$ & $0.0000$ & $2.8$ & $5.4 \times 10^{1}$ & $0.0125$ \\
Descent-Net & $0.0000$ & $0.0000$ & $1.4 \times 10^{-4}$ & $4.9 \times 10^{-6}$ & $0.0019$ \\
\midrule
\multicolumn{6}{c}{\textbf{Number of assets $n=800$}} \\
\midrule
Method & ineq. vio. & eq. vio. & sol. rel. err. & obj. rel. err. & Time (s) \\
\midrule
OSQP        & $0.0000$ & $0.0000$ & $0.0$ & $0.0$ & $0.0207$ \\
HPR-QP      & $0.0000$ & $0.0000$ & $5.0 \times 10^{-6}$ & $6.2 \times 10^{-8}$ & $0.1683$ \\
CuClarabel     & $0.0000$ & $0.0000$ & $5.4 \times 10^{-5}$ & $4.0 \times 10^{-8}$ & $0.1168$ \\
Descent-Net & $0.0000$ & $0.0000$ & $9.3 \times 10^{-4}$ & $7.3 \times 10^{-6}$ & $0.0019$ \\
\midrule
\multicolumn{6}{c}{\textbf{Number of assets $n=4000$}} \\
\midrule
Method & ineq. vio. & eq. vio. & sol. rel. err. & obj. rel. err. & Time (s) \\
\midrule
OSQP        & $0.0000$ & $0.0000$ & $0.0$ & $0.0$ & $0.6024$ \\
HPR-QP      & $0.0000$ & $0.0000$ & $1.6 \times 10^{-5}$ & $2.2 \times 10^{-6}$ & $0.3676$ \\
CuClarabel     & $0.0000$ & $0.0000$ & $4.2 \times 10^{-5}$ & $2.2 \times 10^{-6}$ & $0.5232$ \\
Descent-Net & $0.0000$ & $0.0000$ & $1.6 \times 10^{-4}$ & $1.2 \times 10^{-6}$ & $0.0044$ \\
\bottomrule
\end{tabular}
\end{table}

We find that DC3 fails to produce feasible solutions for $n=800$ and $n=4000$ because its training diverges. This is likely due to DC3's reliance on gradient steps, which are used to enforce inequality constraints, but whose step sizes and momentum decay parameters are difficult to tune for large-scale settings. In contrast, Descent-Net remains accurate and highly efficient across all problem sizes.

The \verb|OSQP| \verb|HPR-QP| and \verb|CuClarabel|'s times report the average runtime for a single instance, whereas the Descent-Net times correspond to the average runtime for a batch of instances. As shown, Descent-Net achieves lower runtimes while maintaining objective errors on the order of $10^{-6}$. The speedups are particularly significant for large-scale instances, demonstrating strong scalability to problems with thousands of variables.

\subsection{AC-OPF}

The objective of the AC optimal power flow (AC-OPF) problem is to determine the optimal power generation that balances supply and demand while satisfying both physical laws and operational constraints of the network. A compact formulation of the AC-OPF problem is as follows:
\begin{equation}
\begin{aligned}
\min_{p_g \in \mathbb{R}^n, \, q_g \in \mathbb{R}^n, \, v \in \mathbb{C}^n} 
    & \quad p_g^\top Q p_g + b^\top p_g  \\[6pt]
\text{s.t.} \quad 
& p_g^{\min} \leqslant p_g \leqslant p_g^{\max}, \quad 
 q_g^{\min} \leqslant q_g \leqslant q_g^{\max}, \quad
 v_m^{\min} \leqslant \lvert v\rvert \leqslant v_m^{\max}, \\ 
& v_a^{\min} \leqslant \angle v_i - \angle v_j \leqslant v_a^{\max}, \quad 
 \lvert v_i(\overline{v_i} - \overline{v_j})\overline{w}_{ij}\rvert \leqslant S_{ij}^{\max}, \\ 
& (p_g - p_d) + (q_g - q_d)i = \mathrm{diag}(v)\overline{W}\,\overline{v}.
\end{aligned}
\end{equation}
Here, $p_d, q_d \in \mathbb{R}^n$ denote the active and reactive power demands, 
and $p_g, q_g \in \mathbb{R}^n$ are the corresponding power generations. 
The complex bus voltage is represented by $v \in \mathbb{C}^n$. 
The nodal admittance matrix $W \in \mathbb{C}^{n \times n}$ encodes the network topology.

Since the equality constraints in this problem are nonlinear, a first-order approximation is not very accurate. As a result, even if the descent direction $d$ is orthogonal to the gradients of all equality constraints, the updated point may still fail to satisfy them. To address this issue, we adopt an equation completion approach, and the details are provided in the appendix \ref{AC-OPF_update}.

We conduct experiments on two AC-OPF problem instances of different scales. Besides D-Proj (i.e., DC3), we use H-Proj \cite{liang2024homeomorphic} as another initialization strategy, and denote the corresponding solutions by $y^D$ and $y^H$. D-Proj originally reduces violations of inequality constraints by performing a gradient descent step on the $\ell_2$ norm of constraint violations. In our experiments, we found that this gradient step is time-consuming and, in practice, often unnecessary. Therefore, we introduce an improved variant of D-Proj by removing the gradient-descent step. This modification significantly reduces the computational time while maintaining comparable satisfaction of the inequality constraints. The optimized initialization obtained using this approach is denoted by $y^{D^\ast}$.

The results in Table~\ref{tab:AC-OPF} indicate that Descent-Net produces solutions with relative objective errors on the order of $10^{-4}$ across all cases, outperforming all neural network–based solvers in accuracy. Moreover, it is approximately four times faster than the standard solvers. The relative error of the solution obtained by Descent-Net decreases only marginally compared to the initial point, which may be due to the non-convex nature of the AC-OPF problem. Note that the runtime of \verb|PYPOWER| is the average per instance since it solves sequentially, while Descent-Net solves instances in parallel, providing much higher efficiency. 

\begin{table}[htbp]
\caption{Results on the AC-OPF task evaluated on the test set with 1024 samples. For methods that support parallel execution, the runtime corresponds to the time required to process one batch with a batch size of 1024.}
\label{tab:AC-OPF}
\resizebox{\textwidth}{!}{%
\centering
\begin{tabular}{lccccc}
\toprule
\multicolumn{6}{c}{\textbf{30-bus system: }$n_{\text{eq}} = 60, n_{\text{ineq}} = 84$} \\
\midrule
Method & ineq. vio. & eq. vio. & sol. rel. err. & obj. rel. err. & Time (s) \\
\midrule
PYPOWER             & $0.0000$ & $0.0000$ & $0.0$ & $0.0$ & $0.2890$ \\
Projection method   & $0.0000$ & $0.0000$ & $5.6\times 10^{-3}$ & $1.7\times 10^{-2}$ & $0.0397$ \\
Warm start          & $0.0000$ & $0.0000$ & $5.5\times 10^{-3}$ & $1.7\times 10^{-2}$ & $0.0393$ \\
D-Proj              & $0.0000$ & $0.0000$ & $5.9\times 10^{-3}$ & $1.9\times 10^{-2}$ & $0.2442$ \\
H-Proj              & $0.0000$ & $0.0000$ & $5.8\times 10^{-3}$ & $1.7\times 10^{-2}$ & $0.2865$ \\
\textbf{Descent-Net ($y_0=y^D$)} & $0.0000$ & $0.0000$ & $4.2\times 10^{-3}$ & $3.6\times 10^{-4}$ & $0.2619$ \\
\textbf{Descent-Net ($y_0=y^H$)} & $0.0000$ & $0.0000$ & $3.5\times 10^{-3}$ & $3.3\times 10^{-4}$ & $0.3039$ \\
\textbf{Descent-Net ($y_0=y^{D^\ast}$)} & $0.0000$ & $0.0000$ & $3.6\times 10^{-3}$ & $2.8\times 10^{-4}$ & $0.0434$ \\
\midrule
\addlinespace[2pt]
\multicolumn{6}{c}{\textbf{118-bus system: }$n_{\text{eq}} = 236, n_{\text{ineq}} = 452$} \\
\midrule
Method & ineq. vio. & eq. vio. & sol. rel. err. & obj. rel. err. & Time (s) \\
\midrule
PYPOWER             & $0.0000$ & $0.0000$ & $0.0$ & $0.0$ & $0.6423$ \\
Projection method   & $0.0000$ & $0.0000$ & $1.5\times 10^{-2}$ & $2.4\times 10^{-3}$ & $0.3040$ \\
Warm start          & $0.0000$ & $0.0000$ & $9.3\times 10^{-3}$ & $1.8\times 10^{-3}$ & $0.3137$ \\
D-Proj              & $0.0000$ & $0.0000$ & $1.3\times 10^{-2}$ & $2.4\times 10^{-3}$ & $0.7542$ \\
H-Proj              & $0.0000$ & $0.0000$ & $1.4\times 10^{-2}$ & $3.1\times 10^{-3}$ & $0.6682$ \\
\textbf{Descent-Net ($y_0=y^D$)} & $0.0000$ & $0.0000$ & $1.2\times 10^{-2}$ & $2.5\times 10^{-4}$ & $0.9480$ \\
\textbf{Descent-Net ($y_0=y^H$)} & $0.0000$ & $0.0000$ & $1.4\times 10^{-2}$ & $7.2\times 10^{-4}$ & $0.8637$ \\
\textbf{Descent-Net ($y_0=y^{D^\ast}$)} & $0.0000$ & $0.0000$ & $2.2\times 10^{-3}$ & $3.0\times 10^{-4}$ & $0.1622$ \\
\bottomrule
\end{tabular}%
}
\end{table}

\section{Conclusion}

In this work, we presented Descent-Net, a neural architecture that incorporates first-order optimality structure for solving constrained optimization problems, with the goal of improving the effectiveness of existing L2O methods. The model generalizes well across different types of instances, including quadratic programs, simple non-convex variants of QPs, and problems with nonlinear constraints such as AC-OPF. In our experiments, Descent-Net produces more accurate solutions than prior L2O approaches while also offering faster inference compared to classical solvers. Moreover, we demonstrate its scalability on large-scale QP problems and portfolio optimization tasks with thousands of assets, highlighting its potential for real-world applications. Future work includes extending the framework to nonlinear constraints and using Descent-Net solutions as warm starts for standard solvers to get high-accurate solutions.

\begin{appendices}

\section{Experiment setting}\label{exp_details}

For the convex QPs and the simple non-convex problems, the parameters are generated as follows. 
The matrix $Q$ is constructed as $Q = R^\top R$, where $R \in \mathbb{R}^{n \times n}$ has i.i.d.\ standard normal entries, ensuring that $Q$ is symmetric and positive semidefinite.  
The constraint matrices $A$ and $G$ are also sampled with i.i.d. standard normal entries.
For each instance, the components of $x$ are drawn i.i.d.\ from the uniform distribution on $[-1,1]$. 
To ensure that the generated problem has a feasible solution, we set $h_i = \sum_j \lvert(G A^{\dagger})_{ij}\rvert $, where $A^{\dagger}$ denotes the Moore--Penrose pseudoinverse of $A$. 
For the AC-OPF experiments, we use the datasets provided in \cite{liang2024homeomorphic}.

For the portfolio optimization experiments, to model slowly evolving asset co-movements, the covariance matrix is fixed across instances and constructed as $\Sigma = R^\top R$, where $R \in \mathbb{R}^{n \times n}$ has i.i.d.\ standard normal entries. 
For each instance, the expected return vector $\mu$ is sampled independently from a uniform distribution over $[0,1]$, and the return threshold $r_{\min}$ is drawn independently: for training, $r_{\min} \sim \mathrm{Uniform}[0.05,0.4]$, while for testing, $r_{\min}$ is taken as a linearly spaced sequence over the same interval.

We summarize the hyperparameters used in our experiments in Table~\ref{tab:hyperparams}. 
Below we briefly describe several important parameters:
\begin{itemize}
    \item $S$: the number of update steps performed in our Descent Net.
    \item $K$: the number of layers within each Descent module, controlling the expressive power of the network.
    \item $\lambda_h$: the penalty factor for equality constraint violations.
    \item $\lambda_g$: the penalty factor for inequality constraint violations.
    \item $q$: the hidden dimension of operator $T$, which specifies the capacity of feature transformation inside each descent step.
    \item $M, \epsilon$: parameters in $c_j$, which is defined in~\eqref{eq:cj}.
\end{itemize}

\begin{table}[htbp]
\centering
\caption{Hyperparameters used in different experiments}
\label{tab:hyperparams}
\begin{tabular}{lcccccccc}
\toprule
\textbf{Problem} 
& \textbf{Epochs} 
& $S$ 
& $K$
& $\lambda_h$
& $\lambda_g$
& $q$
& $M$
& $\epsilon$ \\
\midrule
QP ($n$=100)  
& 150 & 8  & 3 & 5 & 5 & 300 & 1 & 0.0005 \\
QP ($n$=1000)  
& 300 & 5  & 1 & 5 & 5 & 3000 & 1 & 0.0005 \\
QP ($n$=5000)  
& 300 & 8  & 1 & 5 & 5 & 15000 & 1 & 0.0005 \\
Non-convex  
& 150 & 10 & 3 & 5 & 5 & 300 & 1 & 0.0005 \\
AC-OPF (node=30) 
& 300 & 3  & 3 & 5 & 5 & 120 & 1 & 0.0001 \\
AC-OPF (node=118) 
& 300 & 3  & 3 & 5 & 5 & 1080 & 1 & 0.0001 \\
Portfolio ($n$=100) 
& 300 & 3  & 1 & 5 & 5 & 800 & 1 & 0.0001 \\
Portfolio ($n$=800) 
& 300 & 2  & 1 & 5 & 5 & 1200 & 1 & 0.0001 \\
Portfolio ($n$=4000) 
& 300 & 2  & 1 & 5 & 5 & 6000 & 1 & 0.0001 \\
\bottomrule
\end{tabular}
\end{table}

For the experiments other than portfolio optimization, we train the Descent module using the Adam optimizer with an initial learning rate of $0.01$. The learning rate is reduced by a factor of $0.1$ at epochs 50, 100, and 150. The step-size parameter $\beta$ is updated separately using the SGD optimizer with a fixed learning rate of $0.01$. In the AC-OPF experiments, we additionally clip the gradient norm at a threshold of $1$ to stabilize training, following \cite{zhang2019gradient}.

In contrast, the portfolio optimization experiments adopt a different training configuration. Both the Descent module and the step size $\beta$ are optimized using Adam. The initial learning rates are set to $1\times 10^{-3}$ for the Descent module, and to $0.1$, $0.1$, and $0.01$ for $\beta$ in the $n=100$, $n=800$, and $n=4000$ settings, respectively. All learning rates decay by a factor of $0.1$ at epochs 100, 150, and 200 over a total of 300 training epochs.

All experiments were conducted on a server equipped with two AMD EPYC 9754 CPUs (128 cores each, 3.1 GHz) and an NVIDIA H100 GPU.

\section{Effect of Layer number $K$ and Descent steps $S$}

We conducted experiments on the convex quadratic program~\eqref{prob:qp} to evaluate the performance of Descent-Modules with different numbers of layers $K$. The results are shown in Table~\ref{tab:vary_K}. It can be seen that increasing $K$ leads to a slight improvement in performance, but the gains are not significant. Considering computational efficiency, we ultimately choose $K=3$ as the number of layers.

\begin{table}[htbp]
\caption{Performance of Descent-Module with varying $K$}
\label{tab:vary_K}
\centering
\renewcommand{\arraystretch}{1.2}
\begin{tabular}{c cc cc}
\toprule
Layer & ineq. vio. & eq. vio. & sol. rel. err. & obj. rel. err. \\
\midrule
$K=1$  & 0.0000 & 0.0000 & $9.8\times 10^{-2}$ & $2.2\times 10^{-2}$  \\
$K=2$  & 0.0000 & 0.0000 & $9.3\times 10^{-2}$ & $1.8\times 10^{-2}$  \\
$K=3$  & 0.0000 & 0.0000 & $9.2\times 10^{-2}$ & $1.7\times 10^{-2}$  \\
$K=4$  & 0.0000 & 0.0000 & $8.9\times 10^{-2}$ & $1.7\times 10^{-2}$  \\
\bottomrule
\end{tabular}
\end{table}

With $K$ fixed at 3, we further examined the effect of different Descent steps $S$, as summarized in Table~\ref{tab:vary_S}. When the number of update steps is 1, Descent-Net already achieves a solution with a relative error on the order of $10^{-2}$. Increasing the steps to 4 reduces the error to the $10^{-3}$ level, and further increasing to 8 reduces it to the $10^{-4}$ level.

\begin{table}[htbp]
\caption{Performance of Descent-Net with varying $S$}
\label{tab:vary_S}
\centering
\renewcommand{\arraystretch}{1.2}
\begin{tabular}{c cc cc}
\toprule
Descent Step & ineq. vio. & eq. vio. & sol. rel. err. & obj. rel. err. \\
\midrule
$S=1$  & 0.0000 & 0.0000 & $8.9\times 10^{-2}$ & $1.7\times 10^{-2}$  \\
$S=2$  & 0.0000 & 0.0000 & $6.0\times 10^{-2}$ & $1.1\times 10^{-2}$  \\
$S=4$  & 0.0000 & 0.0000 & $3.5\times 10^{-2}$ & $3.4\times 10^{-3}$  \\
$S=6$  & 0.0000 & 0.0000 & $2.5\times 10^{-2}$ & $1.8\times 10^{-3}$  \\
$S=8$  & 0.0000 & 0.0000 & $1.2\times 10^{-2}$ & $3.1\times 10^{-4}$  \\
\bottomrule
\end{tabular}
\end{table}

\section{Step Size Selection Strategies}\label{exp:step size}

We compare the effectiveness of three different step size selection strategies:  
\begin{itemize}
    \item A fixed step size $\alpha = 1/M$;  
    \item The maximum feasible step size $\alpha_{\max}$ that ensures feasibility;  
    \item A learnable scale factor $\sigma(\beta)$ applied to $\alpha_{\max}$.
\end{itemize}

We perform comparative experiments on the convex quadratic program~\eqref{prob:qp}, evaluating three methods based on the feasibility and optimality of their solutions after a fixed number of update steps $S=8$. The corresponding results are presented in Table~\ref{tab:stepsize}. As shown, both the fixed step size $1/M$ and the maximum feasible step size $\alpha_{\max}$ perform worse than our final choice $\alpha=\sigma(\beta)\alpha_{\max}$. The limitation of $1/M$ lies in its lack of flexibility, as a fixed step size cannot adapt to the varying landscape of the problem. And $\sigma(\beta)\alpha_{\max}$ outperforms $\alpha_{\max}$ because the learnable parameter $\beta$ captures useful information that enables a more appropriate scaling of the maximum step size.

\begin{table}[htbp]
\caption{Comparison of different step size selection strategies}
\label{tab:stepsize}
\centering
\renewcommand{\arraystretch}{1.2}
\begin{tabular}{c cc cc}
\toprule
Method & ineq. vio. & eq. vio. & sol. rel. err. & obj. rel. err. \\
\midrule
$1/M$  & 0.0000 & 0.0000 & $9.0\times 10^{-2}$ & $1.2\times 10^{-2}$  \\
$\alpha_{\max}$ & 0.0000 & 0.0000 & $1.1\times 10^{-1}$ & $3.3\times 10^{-2}$\\
$\sigma(\beta) \alpha_{\max}$ & 0.0000 & 0.0000 & $1.2\times 10^{-2}$ & $3.1\times 10^{-4}$  \\
\bottomrule
\end{tabular}
\end{table}

\section{Subproblem}\label{exp:subprob}

In our method, each descent direction $d_s$ is obtained by solving a subproblem~\eqref{eq:subproblem-1}. To assess the ability of the Descent-Net to solve this subproblem, we measure the relative error of the subproblem’s objective value between each layer’s output $d_k$ and the corresponding optimal solution.

We conduct experiments on the convex QP task. For simplicity, we set $S=1$, performing only a single update, and fix the number of Descent-Net layers to $K=3$. We then evaluate the trained network, with the results reported in Table~\ref{tab:subprob}. As shown, the objective value of the subproblem (Descent value) decreases progressively across layers, and by the final layer (layer 3), the relative error in the objective value has already been reduced to $0.001$, demonstrating the efficiency of the Descent-Net in solving the subproblem.

\begin{table}[htbp]
\caption{Effectiveness of Descent-Net in solving subproblem}
\label{tab:subprob}
\centering
\begin{tabular}{ccc}
\toprule
Layer & Descent Value & Relative Error \\
\midrule
0    & 1740.4817   & 2.6463    \\
1    & 505.2591   & 0.0585    \\
2    & 478.2721   & 0.0020    \\
3    & 477.7893   & 0.0010    \\
\bottomrule
\end{tabular}
\label{tab:descent_errors}
\end{table}

\section{Learnable $\gamma$ in Descent-Net}\label{exp:gamma}

We recorded the values of the learnable parameter $\gamma$ in each layer of the $S$ Descent Modules of the trained Descent-Net. For both QP and Nonconvex problems, $\gamma$ is initialized to $0.1$, while for the AC-OPF problem it is initialized to $1$. The results are presented in Table~\ref{tab:gamma_simple}, Table~\ref{tab:gamma_nonconvex} and Table~\ref{tab:AC-OPF_gamma}. These results indicate that the network is able to adjust $\gamma$ dynamically across layers. In many cases, the values of $\gamma$ tend to decrease with the layer depth, which is consistent with the requirement of diminishing step sizes for convergence in subgradient methods.

\begin{table}[htbp]
\centering
\caption{Values of $\gamma$ in Descent-Net across some steps for Simple QP problems}
\label{tab:gamma_simple}
\begin{tabular}{lcccc}
\toprule
 & Step 1 & Step 3 & Step 5 & Step 7 \\
\midrule
$\gamma_1$ & $1.15 \times 10^{-2}$ & $6.37 \times 10^{-3}$ & $1.03 \times 10^{-2}$ & $1.75 \times 10^{-2}$ \\
$\gamma_2$ & $1.41 \times 10^{-2}$ & $5.39 \times 10^{-2}$ & $5.71 \times 10^{-3}$ & $4.49 \times 10^{-3}$ \\
$\gamma_3$ & $9.85 \times 10^{-2}$ & $1.00 \times 10^{-1}$ & $9.93 \times 10^{-4}$ & $2.23 \times 10^{-3}$ \\
\bottomrule
\end{tabular}
\end{table}

\begin{table}[htbp]
\centering
\caption{Values of $\gamma$ in Descent-Net across some steps for Nonconvex problems}
\label{tab:gamma_nonconvex}
\begin{tabular}{lccccc}
\toprule
 & Step 1 & Step 3 & Step 5 & Step 7 & Step 9 \\
\midrule
$\gamma_1$ & $4.79 \times 10^{-2}$ & $1.02 \times 10^{-2}$ & $3.40 \times 10^{-3}$ & $8.22 \times 10^{-3}$ & $1.95 \times 10^{-2}$ \\
$\gamma_2$ & $4.92 \times 10^{-2}$ & $9.53 \times 10^{-2}$ & $4.39 \times 10^{-4}$ & $3.08 \times 10^{-3}$ & $4.78 \times 10^{-3}$ \\
$\gamma_3$ & $1.19 \times 10^{-1}$ & $5.53 \times 10^{-4}$ & $3.53 \times 10^{-2}$ & $1.94 \times 10^{-3}$ & $3.25 \times 10^{-3}$ \\
\bottomrule
\end{tabular}
\end{table}

\begin{table}[htbp]
\centering
\caption{Values of $\gamma$ in Descent-Net across steps for AC-OPF problems}
\label{tab:AC-OPF_gamma}
\begin{tabular}{lccc @{\hskip 1cm} lccc}
\toprule
\multicolumn{4}{c}{\textbf{node = 30, H-Proj}} & \multicolumn{4}{c}{\textbf{node = 30, D-Proj}} \\
\midrule
& Step 1 & Step 2 & Step 3 & & Step 1 & Step 2 & Step 3 \\
\midrule
$\gamma_1$ & 1.00 & 1.00 & 1.00 & $\gamma_1$ & 1.00 & 1.00 & 1.00 \\
$\gamma_2$ & 0.99 & 1.01 & 0.99 & $\gamma_2$ & 0.99 & 0.99 & 1.00 \\
$\gamma_3$ & 1.10 & 0.01 & 0.84 & $\gamma_3$ & 1.42 & 0.20 & 1.01 \\
\midrule
\multicolumn{4}{c}{\textbf{node = 118, H-Proj}} & \multicolumn{4}{c}{\textbf{node = 118, D-Proj}} \\
\midrule
& Step 1 & Step 2 & Step 3 & & Step 1 & Step 2 & Step 3 \\
\midrule
$\gamma_1$ & 1.00 & 1.00 & 1.00 & $\gamma_1$ & 1.00 & 1.00 & 1.00 \\
$\gamma_2$ & 1.00 & 1.00 & 1.00 & $\gamma_2$ & 1.00 & 1.00 & 1.00 \\
$\gamma_3$ & 1.07 & 1.06 & 1.10 & $\gamma_3$ & 1.29 & 1.03 & 0.99 \\
\bottomrule
\end{tabular}
\end{table}

\section{Comparison with PGM (Projected Subgradient Method)}\label{exp:pgm}

We compare Descent-Net with the original PGM. Specifically, we remove the operator $T^k$ in Descent-Net so that each layer reduces to \eqref{pgm}. We still treat the step size $\gamma_k$ as a learnable parameter and train this degenerated network in the same manner as Descent-Net. 

We evaluate the performance under different numbers of iterations $K$, with the results reported in Table~\ref{tab:pgm}. We observe that PGM is inefficient, as the relative error in the objective value compared to the initial solution decreases very little with increasing iterations. This is likely due to the difficulty of selecting an appropriate step size for PGM. In contrast, the Descent-Net achieves strong solution quality with only three layers, which also leads to a significant advantage in computational efficiency.

\begin{table}[htbp]
\caption{Comparison of Descent-Net and PGM on the convex QP task.}
\label{tab:pgm}
\centering
\begin{tabular}{lccccc}
\toprule
Method & ineq. vio. & eq. vio. & sol. rel. err. & obj. rel. err. & Time (s) \\
\midrule
\textbf{PGM ($K=10$)} & $0.0000$ & $0.0000$ & $1.9 \times 10^{-1}$ & $1.1 \times 10^{-1}$ & $0.0270$ \\
\textbf{PGM ($K=20$)} & $0.0000$ & $0.0000$ & $1.9 \times 10^{-1}$ & $1.1 \times 10^{-1}$ & $0.0501$ \\
\textbf{PGM ($K=50$)} & $0.0000$ & $0.0000$ & $1.9 \times 10^{-1}$ & $1.1 \times 10^{-1}$ & $0.1119$ \\
\textbf{Descent-Net ($K=3$)} & $0.0000$ & $0.0000$ & $1.2\times 10^{-2}$ & $3.1\times 10^{-4}$ & $0.0132$ \\
\bottomrule
\end{tabular}
\end{table}

\section{Descent Updates in the AC-OPF Problem}\label{AC-OPF_update}

In the AC-OPF problem, given $(n-m)$ entries of a feasible point $y \in \mathbb{R}^n$, 
the remaining $m$ entries are, in general, determined by the $m$ equality constraints $h_x(y)=0$. 

Following the method in \cite{dontidc3,liang2024homeomorphic,wu2025constraint}, we assume the existence of a function $\varphi_x: \mathbb{R}^{n-m} \rightarrow \mathbb{R}^m$ such that $h_x([z, \varphi_x(z)]) = 0$. 
This allows us to eliminate the equality constraints and reformulate the problem in terms of the partial variable $z$. 
We can then perform descent direction updates on $z$, where the optimization problem involves only the inequality constraints:
\begin{equation}
    \min_{z \in \mathbb{R}^{n-m}} \tilde{f}_x(z), 
    \quad \text{s.t.} \quad \tilde{g}_x(z) \leqslant 0,
\end{equation}
where 
$\tilde{f}_x(z) = f_x\left([z^T, \varphi_x(z)^T]^T\right)$ 
and 
$\tilde{g}_x(z) = g_x\left([z^T, \varphi_x(z)^T]^T\right)$.

Using the chain rule, we can compute the derivative of $\varphi_x$ with respect to $z$, 
even without an explicit expression of $\varphi_x$:
\begin{equation*}
\begin{aligned}
    0 &= \frac{\mathrm{d}}{\mathrm{d}z} h_x \big( \varphi_x(z) \big) 
       = \frac{\partial h_x}{\partial z} 
       + \frac{\partial h_x}{\partial \varphi_x(z)} \frac{\partial \varphi_x(z)}{\partial z} \\
      &= J^{h}_{:,0:m} + J^{h}_{:,m:n} \frac{\partial \varphi_x(z)}{\partial z}, \\
    \Rightarrow & \quad 
    \frac{\partial \varphi_x(z)}{\partial z} 
    = -\big(J^{h}_{:,m:n}\big)^{-1} J^{h}_{:,0:m}.
\end{aligned}
\end{equation*}
Here, $J^h \in \mathbb{R}^{m \times n}$ denotes the Jacobian matrix of the equality constraints $h_x(y)$ with respect to $y$. The notation $J^{h}_{:,0:m}$ and $J^{h}_{:,m:n}$ represents the submatrices corresponding to the partial derivatives with respect to $z$ and $\varphi_x(z)$, respectively.

From this result, we can further obtain the gradients of the objective and inequality constraints with respect to $z$. These gradient informations are then passed to the Descent-Net, which outputs the descent direction $d_z$ for the partial variable $z$. 

In order to obtain the complete descent direction 
$d = [d_z, d_{\varphi}]$ for $y$, we also need the expression of $d_{\varphi}$. 
To ensure that the equality constraints remain satisfied, we require the following
\begin{equation*}
    \begin{aligned}
        h(z + \alpha d_z, \varphi(z) + \alpha d_{\varphi})
    \approx& h\big(z,\varphi(z)\big) + \alpha\,J^{h}
    \begin{bmatrix} d_z \\[2pt] d_{\varphi} \end{bmatrix}\\
    =& h\big(z,\varphi(z)\big) + \alpha \big( J^{h}_{:,0:m} d_z + J^{h}_{:,m:n} d_{\varphi} \big) = 0,
    \end{aligned}
\end{equation*}
where $\alpha>0$ is the step size. Hence, we obtain
\begin{equation*}
    d_{\varphi} 
    = -\big(J^{h}_{:,m:n}\big)^{-1} J^{h}_{:,0:m}\, d_z 
    \;-\; \big(J^{h}_{:,m:n}\big)^{-1}\tfrac{h\big(z,\varphi(z)\big)}{\alpha}.
\end{equation*}

\end{appendices}


\bibliographystyle{unsrt}
\bibliography{ref}

@article{liang2024homeomorphic,
  title={Homeomorphic projection to ensure neural-network solution feasibility for constrained optimization},
  author={Liang, Enming and Chen, Minghua and Low, Steven H},
  journal={Journal of Machine Learning Research},
  volume={25},
  number={329},
  pages={1--55},
  year={2024}
}

@inproceedings{dontidc3,
  title={DC3: A learning method for optimization with hard constraints},
  author={Donti, Priya L and Rolnick, David and Kolter, J Zico},
  booktitle={International Conference on Learning Representations},
  year={2021}
}

@misc{CuClarabel,
      title={CuClarabel: GPU Acceleration for a Conic Optimization Solver}, 
      author={Yuwen Chen and Danny Tse and Parth Nobel and Paul Goulart and Stephen Boyd},
      year={2024},
      eprint={2412.19027},
      archivePrefix={arXiv},
      primaryClass={math.OC},
      url={https://arxiv.org/abs/2412.19027}, 
}

@book{zoutendijk1960methods,
  title={Methods of feasible directions},
  author={Zoutendijk, G},
  volume={960},
  year={1960},
  publisher={Elsevier},
  address={Amsterdam}  
}

@article{geng2020coercing,
  title={Coercing machine learning to output physically accurate results},
  author={Geng, Zhenglin and Johnson, Daniel and Fedkiw, Ronald},
  journal={Journal of Computational Physics},
  volume={406},
  pages={109099},
  year={2020},
  publisher={Elsevier}
}

@book{luenberger1984linear,
  title={Linear and nonlinear programming},
  author={Luenberger, David G and Ye, Yinyu and others},
  volume={2},
  year={1984},
  publisher={Springer},
  address={New York, NY}  
}

@inproceedings{li2024pdhg,
  title={PDHG-unrolled learning-to-optimize method for large-scale linear programming},
  author={Li, Bingheng and Yang, Linxin and Chen, Yupeng and Wang, Senmiao and Mao, Haitao and Chen, Qian and Ma, Yao and Wang, Akang and Ding, Tian and Tang, Jiliang and others},
  booktitle={Proceedings of the 41st International Conference on Machine Learning},
  pages={29164--29180},
  year={2024}
}

@article{wu2024designing,
  title={Designing Universally-Approximating Deep Neural Networks: A First-Order Optimization Approach},
  author={Wu, Zhoutong and Xiao, Mingqing and Fang, Cong and Lin, Zhouchen},
  journal={IEEE Transactions on Pattern Analysis and Machine Intelligence},
  year={2024},
  publisher={IEEE}
}

@article{agrawal2019differentiable,
  title={Differentiable convex optimization layers},
  author={Agrawal, Akshay and Amos, Brandon and Barratt, Shane and Boyd, Stephen and Diamond, Steven and Kolter, J Zico},
  journal={Advances in neural information processing systems},
  volume={32},
  year={2019}
}

@article{donti2017task,
  title={Task-based end-to-end model learning in stochastic optimization},
  author={Donti, Priya and Amos, Brandon and Kolter, J Zico},
  journal={Advances in neural information processing systems},
  volume={30},
  year={2017}
}

@inproceedings{yangprojection,
  title={Projection-Based Constrained Policy Optimization},
  author={Yang, Tsung-Yen and Rosca, Justinian and Narasimhan, Karthik and Ramadge, Peter J},
  booktitle={International Conference on Learning Representations},
  year={2020}
}

@inproceedings{amos2017optnet,
  title={Optnet: Differentiable optimization as a layer in neural networks},
  author={Amos, Brandon and Kolter, J Zico},
  booktitle={International conference on machine learning},
  pages={136--145},
  year={2017},
  organization={PMLR}
}

@article{stellato2020osqp,
  title={OSQP: An operator splitting solver for quadratic programs},
  author={Stellato, Bartolomeo and Banjac, Goran and Goulart, Paul and Bemporad, Alberto and Boyd, Stephen},
  journal={Mathematical Programming Computation},
  volume={12},
  number={4},
  pages={637--672},
  year={2020},
  publisher={Springer}
}

@article{wachter2006implementation,
  title={On the implementation of an interior-point filter line-search algorithm for large-scale nonlinear programming},
  author={W{\"a}chter, Andreas and Biegler, Lorenz T},
  journal={Mathematical programming},
  volume={106},
  pages={25--57},
  year={2006},
  publisher={Springer}
}

@misc{zimmerman2005matlab,
  title={A MATLAB power system simulation package},
  author={Zimmerman, Ray D and Murillo-S{\'a}nchez, Carlos E and Gan, Deqiang},
  journal={Ithaca, NY: Cornell University, 2011 [2011-12-14]. http://www. pserc. cornell. edu/matpower},
  year={2005}
}

@article{pironneau1973rate,
  title={Rate of convergence of a class of methods of feasible directions},
  author={Pironneau, O and Polak, E},
  journal={SIAM Journal on Numerical Analysis},
  volume={10},
  number={1},
  pages={161--174},
  year={1973},
  publisher={SIAM}
}

@article{topkis1967convergence,
  title={On the convergence of some feasible direction algorithms for nonlinear programming},
  author={Topkis, Donald M and Veinott, Jr, Arthur F},
  journal={SIAM Journal on Control},
  volume={5},
  number={2},
  pages={268--279},
  year={1967},
  publisher={SIAM}
}

@article{wu2025constraint,
  title={Constraint Boundary Wandering Framework: Enhancing Constrained Optimization With Deep Neural Networks},
  author={Wu, Shuang and Chen, Shixiang and Shen, Li and Zhang, Lefei and Tao, Dacheng},
  journal={IEEE Transactions on Pattern Analysis \& Machine Intelligence},
  volume={47},
  number={08},
  pages={6369--6381},
  year={2025},
  publisher={IEEE Computer Society}
}

@article{bengio2021machine,
  title={Machine learning for combinatorial optimization: a methodological tour d’horizon},
  author={Bengio, Yoshua and Lodi, Andrea and Prouvost, Antoine},
  journal={European Journal of Operational Research},
  volume={290},
  number={2},
  pages={405--421},
  year={2021},
  publisher={Elsevier}
}

@article{chen2022learning,
  title={Learning to optimize: A primer and a benchmark},
  author={Chen, Tianlong and Chen, Xiaohan and Chen, Wuyang and Heaton, Howard and Liu, Jialin and Wang, Zhangyang and Yin, Wotao},
  journal={Journal of Machine Learning Research},
  volume={23},
  number={189},
  pages={1--59},
  year={2022}
}

@article{andrychowicz2016learning,
  title={Learning to learn by gradient descent by gradient descent},
  author={Andrychowicz, Marcin and Denil, Misha and Gomez, Sergio and Hoffman, Matthew W and Pfau, David and Schaul, Tom and Shillingford, Brendan and De Freitas, Nando},
  journal={Advances in neural information processing systems},
  volume={29},
  year={2016}
}

@article{chen2018theoretical,
  title={Theoretical linear convergence of unfolded ISTA and its practical weights and thresholds},
  author={Chen, Xiaohan and Liu, Jialin and Wang, Zhangyang and Yin, Wotao},
  journal={Advances in Neural Information Processing Systems},
  volume={31},
  year={2018}
}

@article{boyd2011distributed,
  title={Distributed optimization and statistical learning via the alternating direction method of multipliers},
  author={Boyd, Stephen and Parikh, Neal and Chu, Eric and Peleato, Borja and Eckstein, Jonathan and others},
  journal={Foundations and Trends{\textregistered} in Machine learning},
  volume={3},
  number={1},
  pages={1--122},
  year={2011},
  publisher={Now Publishers, Inc.}
}

@book{faigle2013algorithmic,
  title={Algorithmic principles of mathematical programming},
  author={Faigle, Ulrich and Kern, Walter and Still, Georg},
  volume={24},
  year={2013},
  publisher={Springer},
  address={Dordrecht}  
}

@book{nocedal1999numerical,
  title={Numerical optimization},
  author={Nocedal, Jorge and Wright, Stephen J},
  year={1999},
  publisher={Springer},
  address={New York, NY}  
}

@inproceedings{zhang2018ista,
  title={ISTA-Net: Interpretable optimization-inspired deep network for image compressive sensing},
  author={Zhang, Jian and Ghanem, Bernard},
  booktitle={Proceedings of the IEEE conference on computer vision and pattern recognition},
  pages={1828--1837},
  year={2018}
}

@article{bolte2021nonsmooth,
  title={Nonsmooth implicit differentiation for machine-learning and optimization},
  author={Bolte, J{\'e}r{\^o}me and Le, Tam and Pauwels, Edouard and Silveti-Falls, Tony},
  journal={Advances in neural information processing systems},
  volume={34},
  pages={13537--13549},
  year={2021}
}

@article{berthet2020learning,
  title={Learning with differentiable pertubed optimizers},
  author={Berthet, Quentin and Blondel, Mathieu and Teboul, Olivier and Cuturi, Marco and Vert, Jean-Philippe and Bach, Francis},
  journal={Advances in neural information processing systems},
  volume={33},
  pages={9508--9519},
  year={2020}
}

@inproceedings{poganvcic2019differentiation,
  title={Differentiation of blackbox combinatorial solvers},
  author={Pogan{\v{c}}i{\'c}, Marin Vlastelica and Paulus, Anselm and Musil, Vit and Martius, Georg and Rolinek, Michal},
  booktitle={International Conference on Learning Representations},
  year={2019}
}

@article{gao2022learning,
  title={Learning to Optimize on Riemannian Manifolds},
  author={Gao, Zhi and Wu, Yuwei and Fan, Xiaomeng and Harandi, Mehrtash and Jia, Yunde},
  journal={IEEE Transactions on Pattern Analysis and Machine Intelligence},
  volume={45},
  number={5},
  pages={5935--5952},
  year={2022},
  publisher={IEEE}
}

@article{bai2019deep,
  title={Deep equilibrium models},
  author={Bai, Shaojie and Kolter, J Zico and Koltun, Vladlen},
  journal={Advances in Neural Information Processing Systems},
  volume={32},
  year={2019}
}

@article{chen2018neural,
  title={Neural ordinary differential equations},
  author={Chen, Ricky TQ and Rubanova, Yulia and Bettencourt, Jesse and Duvenaud, David K},
  journal={Advances in neural information processing systems},
  volume={31},
  year={2018}
}

@misc{graves2014generating,
      title={Generating Sequences With Recurrent Neural Networks}, 
      author={Alex Graves},
      year={2014},
      eprint={1308.0850},
      archivePrefix={arXiv},
      primaryClass={cs.NE}
}

@inproceedings{xie2019differentiable,
  title={Differentiable linearized ADMM},
  author={Xie, Xingyu and Wu, Jianlong and Liu, Guangcan and Zhong, Zhisheng and Lin, Zhouchen},
  booktitle={International Conference on Machine Learning},
  pages={6902--6911},
  year={2019},
  organization={PMLR}
}

@inproceedings{wang2023linsatnet,
  title={LinSATNet: the positive linear satisfiability neural networks},
  author={Wang, Runzhong and Zhang, Yunhao and Guo, Ziao and Chen, Tianyi and Yang, Xiaokang and Yan, Junchi},
  booktitle={International Conference on Machine Learning},
  pages={36605--36625},
  year={2023},
  organization={PMLR}
}

@article{zhang2019gradient,
  title={Why gradient clipping accelerates training: A theoretical justification for adaptivity},
  author={Zhang, Jingzhao and He, Tianxing and Sra, Suvrit and Jadbabaie, Ali},
  journal={arXiv preprint arXiv:1905.11881},
  year={2019}
}

@inproceedings{chen2021enforcing,
  title={Enforcing policy feasibility constraints through differentiable projection for energy optimization},
  author={Chen, Bingqing and Donti, Priya L and Baker, Kyri and Kolter, J Zico and Berg{\'e}s, Mario},
  booktitle={Proceedings of the Twelfth ACM International Conference on Future Energy Systems},
  pages={199--210},
  year={2021}
}

@inproceedings{diehl2019warm,
  title={Warm-starting AC optimal power flow with graph neural networks},
  author={Diehl, Frederik},
  booktitle={33rd Conference on Neural Information Processing Systems (NeurIPS 2019)},
  pages={1--6},
  year={2019}
}

@inproceedings{baker2019learning,
  title={Learning warm-start points for AC optimal power flow},
  author={Baker, Kyri},
  booktitle={2019 IEEE 29th International Workshop on Machine Learning for Signal Processing (MLSP)},
  pages={1--6},
  year={2019},
  organization={IEEE}
}

@article{tan2020learning,
  title={Learning linear programs from optimal decisions},
  author={Tan, Yingcong and Terekhov, Daria and Delong, Andrew},
  journal={Advances in Neural Information Processing Systems},
  volume={33},
  pages={19738--19749},
  year={2020}
}

@article{fabozzi2008portfolio,
  title={Portfolio selection},
  author={Fabozzi, Frank J and Markowitz, Harry M and Gupta, Francis},
  journal={Handbook of finance},
  volume={2},
  pages={3--13},
  year={2008},
  publisher={American Cancer Society}
}

@article{chen2025hpr,
  title={HPR-QP: A dual Halpern Peaceman-Rachford method for solving large-scale convex composite quadratic programming},
  author={Chen, Kaihuang and Sun, Defeng and Yuan, Yancheng and Zhang, Guojun and Zhao, Xinyuan},
  journal={arXiv preprint arXiv:2507.02470},
  year={2025}
}

\end{document}